\documentclass{amsart}

\usepackage{begnac}
\usepackage{stmaryrd}
\usepackage{pictexwd,dcpic}

\newcommand{\ra}{{\mathbf{a}}}
\newcommand{\rb}{{\mathbf{b}}}
\newcommand{\rc}{{\mathbf{c}}}

\newcommand{\rp}{{\mathbf{p}}}
\renewcommand{\rq}{{\mathbf{q}}}
\newcommand{\rr}{{\mathbf{r}}}
\newcommand{\rx}{{\mathbf{x}}}
\newcommand{\ry}{{\mathbf{y}}}
\newcommand{\rz}{{\mathbf{z}}}
\newcommand{\rA}{{\mathbf{A}}}

\newcommand{\rM}{{\mathbf{M}}}

\newcommand{\rT}{{\mathbf{T}}}

\newcommand{\rcM}{\mathbold{\cM}}

\newcommand{\llangle}{\langle\!\!\langle}
\newcommand{\rrangle}{\rangle\!\!\rangle}
\renewcommand{\[}{\llbracket}
\renewcommand{\]}{\rrbracket}
\newcommand{\<}{\llangle}
\renewcommand{\>}{\rrangle}

\begin{document}

\title{On theories of random variables}

\author{Itaï Ben Yaacov}

\address{Itaï \textsc{Ben Yaacov} \\
  Université Claude Bernard -- Lyon 1 \\
  Institut Camille Jordan, CNRS UMR 5208 \\
  43 boulevard du 11 novembre 1918 \\
  69622 Villeurbanne Cedex \\
  France}

\urladdr{\url{http://math.univ-lyon1.fr/~begnac/}}

\thanks{Research supported by
  ANR chaire d'excellence junior THEMODMET (ANR-06-CEXC-007) and
  by Marie Curie research network ModNet}

\thanks{The author wishes to thank the referee for many helpful remarks regarding the structure of the articles and references}

\svnInfo $Id: RandVar.tex 1281 2011-10-25 08:16:50Z begnac $
\thanks{\textit{Revision} {\svnInfoRevision} \textit{of} \today}

%\date{\today}
\keywords{random variables; continuous logic; metric structures}
%\subjclass[2000]{}

\begin{abstract}
  Nous étudions des théories d'espaces de variables aléatoires~:
  en un premier temps, nous considérons les variables aléatoires à
  valeurs dans l'intervalle $[0,1]$, puis à valeur dans des structures
  métriques quelconques,
  généralisant la procédure d'aléatoirisation de structures
  classiques due à Keisler.
  Nous démontrons des résultats de préservation et de non-préservation
  de propriétés modèle-théoriques par cette construction~:
  \begin{enumerate}
  \item L'aléatoirisée d'une structure ou théorie stable est stable.
  \item L'aléatoirisée d'une structure ou théorie simple instable
    n'est pas simple.
  \end{enumerate}
  Nous démontrons également que dans la structure aléatoirisée,
  tout type est un type de Lascar.

  \bigskip \noindent
  We study theories of spaces of random variables:
  first, we consider random variables with values in the interval
  $[0,1]$, then with values in an arbitrary metric structure,
  generalising Keisler's randomisation of classical structures.
  We prove preservation and non-preservation results for model
  theoretic properties under this construction:
  \begin{enumerate}
  \item The randomisation of a stable structure is stable.
  \item The randomisation of a simple unstable structure is not
    simple.
  \end{enumerate}
  We also prove that in the randomised structure, every type is a
  Lascar type.
\end{abstract}

\maketitle

\section*{Introduction}

Mathematical structures arising in the theory of probabilities are among the most natural examples for metric structures which admit a model theoretic treatment, albeit not in the strict setting of classical first order logic.
Examples include the treatment of adapted spaces by Keisler \& Fajardo \cite{Fajardo-Keisler:StochasticProcesses}, in which no logic of any kind appears explicitly (even though many model theoretic notions, such as types, do appear).
Another example, which is the main topic of the present paper, is Keisler's randomisation construction \cite{Keisler:Randomizing}, in which one considers spaces of random variables whose values lie in some given structures.
The randomisation construction was originally set up in the formalism of classical first order logic, representing the probability space underlying the randomisation by its probability algebra, namely, the Boolean algebra of events up to null measure (defined abstractly, a \emph{probability algebra} is a measure algebra of total mass one, see Fremlin \cite{Fremlin:MeasureTheoryVol3}).
We consider that this formalism was not entirely adequate for the purpose, since the class of probability algebras is not elementary in classical first order logic, a fact which restricts considerably what can be done or proved (for example, the randomised structure interprets an atomless Boolean algebra, and can therefore be neither dependent nor simple).
To the best of our knowledge, the first model theoretic treatment of a probabilistic structure in which notions such as stability and model theoretic independence were considered was carried out by the author in \cite{BenYaacov:SchroedingersCat}, for the class of probability algebras, in the formalism of compact abstract theories.
While this latter formalism was adequate, in the sense that it did allow one to show that probability algebras are stable and that the model theoretic independence coincides with the probabilistic one, it was quite cumbersome, and soon to become obsolete.

Continuous first order logic is a relatively new formalism, at least in its present form, proposed by Alexander Usvyatsov and the author \cite{BenYaacov-Usvyatsov:CFO} for model theoretic treatment of (classes of) complete metric structures.
For example, we observe there that the class of probability algebras is elementary, its theory admitting a simple set of axioms, and that the theory of atomless probability algebras admits quantifier elimination, thus simplifying considerably many of the technical considerations contained in \cite{BenYaacov:SchroedingersCat}.
Viewing probability algebras as metric structures in this fashion, rather than as classical structures, allowed Keisler and the author \cite{BenYaacov-Keisler:MetricRandom} to present the randomisation as a metric structure, and we contend that this metric randomisation is the ``correct'' one.
Arguments to this effect include several preservation results which would be false in the formalism of \cite{Keisler:Randomizing}.
For example, in \cite{BenYaacov-Keisler:MetricRandom} we prove that if a structure is stable then so is its randomisation, while preservation of dependence was proved by the author in \cite{BenYaacov:RandomVC}.
Another argument, both æsthetic and practical, is that types in the metric randomisation are very natural objects, namely regular Borel probability measures on the space of types of the original theory, also referred to nowadays as \emph{Keisler measures} \cite{Keisler:MeasuresAndForking}, and which turn out to be particularly useful for the study of dependent theories, e.g., in \cite{Hrushovski-Peterzil-Pillay:GroupsMeasureNIP}.

We still find the current state of knowledge, and existing treatment, of randomisation, wanting on several points.
First, since the randomisation of a discrete structure (or theory) necessarily produces a metric one, the question of randomising metric structures arises quite naturally.
In fact, it is quite easy to construct the randomisation of a metric structure (or theory) indirectly, by letting its type spaces be the spaces of regular Borel probability measures as mentioned above, a fact which was used in \cite{BenYaacov:RandomVC} to point out that the preservation of dependence holds for the randomisation of metric structures as well, even though the latter had not yet been formally defined.
However, the point of view of theories as type spaces, while a  personal favourite of the author (see for example \cite{BenYaacov:PositiveModelTheoryAndCats}), is far from being universally accepted, creating the need for an ``ordinary'' construction of the randomisation of a metric structure, with a natural language, axioms, and all.
A second point is that the treatment of randomisation in \cite{BenYaacov-Keisler:MetricRandom} relies greatly on \cite{Keisler:Randomizing}, many times referring to it for proofs, even though some fundamental aspects of the set-up are different, requiring the reader to continually verify that the arguments do transfer.

The aim of the present paper is to remedy these shortcomings by providing a self-contained treatment of randomisation in the metric setting, and show (or point out) that the preservation results of \cite{BenYaacov-Keisler:MetricRandom,BenYaacov:RandomVC} hold in the metric setting as well.
In addition, we turn the preservation of dependence into a dichotomy by showing that if $T$ is \emph{not} dependent then its randomisation $T^R$ cannot even be simple, and in fact has $TP_2$.
We also improve a corollary of the preservation of stability of \cite{BenYaacov-Keisler:MetricRandom}, namely that in randomised stable structures types over sets are Lascar types, proving the same for arbitrary randomised structures.
As a minor point, we simplify the language (and theory), and rather than name in $\cL^R$ (the randomisation language) the randomisation $\[\varphi\]$ of each $\cL$-formula $\varphi$, we name the function symbols and the randomisations of the relation symbols of $\cL$ alone.

The paper is organised as follows.
In \fref{sec:Proofs} we consider formal deductions in propositional continuous logic, after Rose, Rosser and Church.
These are used in \fref{sec:AxRV} to give axioms for the theory of spaces $[0,1]$-valued random variables, which play the role played by probability algebras in \cite{BenYaacov-Keisler:MetricRandom}.
Model theoretic properties of this theory are deduced from those of the theory of probability algebras, with which it is biïnterpretable.
In \fref{sec:Randomisation} we define and study the randomisations of metric structures, namely spaces of random variables whose values lie in metric structures.
We give axioms for the theory of these random structures, prove quantifier elimination in the appropriate language, characterise types and so on.
We also prove a version of Łoś's Theorem for randomisations, in which the ultra-filter is replaced with an arbitrary integration functional.
In \fref{sec:Preservation} we prove several preservation and non preservation results.
In \fref{sec:LascarTypes} we prove that in random structures, types over sets are Lascar types, so in the stable case they are stationary.

\section{On results of Rose, Rosser and Chang}
\label{sec:Proofs}

In the late 1950s Rose and Rosser \cite{Rose-Rosser:FragmentsManyValuedCalculi} proved the completeness of a proof system for Łukasiewicz's many-valued propositional logic, subsequently improved and simplified by Chang \cite{Chang:ProofOfLukasiewiczAxiom,Chang:AlgebraicAnalysisMVLogic,Chang:NewCompletenessProof}.
This logic is very close to propositional continuous logic.
Syntactically, the notation is quite different, partially stemming from the fact we identify \emph{True} with $0$, rather than with $1$.
Also, the connective $\half$ does not exist in Łukasiewicz's logic.
Semantically, we only allow the standard unit interval $[0,1]$ as a set of truth values, while some fuzzy logicians allow non-standard extensions thereof (namely, they allow infinitesimal truth values).
We should therefore be careful in how we use their results.

In these references, Propositional Łukasiewicz Logic is presented
using Polish (prefix) notation, without parentheses.
A formula is either an atomic proposition, $C\varphi\psi$ or $N\varphi$,
where $\varphi$ and $\psi$ are simpler formulae.
We shall prefer to use the notation of continuous logic, replacing
$C\varphi\psi$ with $\psi \dotminus \varphi$ and $N\varphi$ with $\neg\varphi$.

\begin{dfn}
  \label{dfn:LukasiewiczLogic}
  Let $\cS_0 = \{P_i\colon i \in I\}$ be a set distinct symbols, which we view as atomic proposition.
  Let $\cS$ be freely generated from $\cS_0$ with the formal binary operation $\dotminus$ and unary operation $\neg$.
  Then $\cS$ is a \emph{Łukasiewicz logic}.
\end{dfn}

\begin{dfn}
  Let $\cS$ be a Łukasiewicz logic.
  \begin{enumerate}
  \item
    For every map $v_0\colon \cS_0 \to [0,1]$, let $v\colon \cS \to [0,1]$ be the unique map extending $v_0$ such that $v(\varphi\dotminus\psi) = v(\varphi) \dotminus v(\psi)$ and $v(\neg\varphi) = 1-v(\varphi)$.
    We call $v$ the \emph{truth assignment} defined by $v_0$, and $v(\varphi)$ is the \emph{truth value} of $\varphi$.
  \item
    If $v(\varphi) = v(\psi)$ for every truth assignment $v$, we say that $\varphi$ and $\psi$ are \emph{equivalent}, and write $\varphi \equiv \psi$.
  \item
    If $v(\varphi) = 0$ we say that $v$ is a \emph{model} of $\varphi$, in symbols $v \vDash \varphi$.
    If $\Sigma \subseteq \cS$, then $v \vDash \Sigma$ if $v \vDash \varphi$ for all $\varphi \in \Sigma$.
    We say that $\varphi$ (or $\Sigma$) is \emph{satisfiable} if it has a model.
  \item
    Let $\Sigma \subseteq \cS$ and $\varphi \in \cS$.
    We say that $\Sigma$ \emph{entails} $\varphi$, or that $\varphi$ is a \emph{logical consequence} of $\Sigma$, if every model of $\Sigma$ is a model of $\varphi$.
    This is denoted $\Sigma \vDash \varphi$.
  \end{enumerate}
  When we wish to make the ambient logic explicit we may write $\vDash_\cS$, $\equiv_\cS$, and so on.
\end{dfn}

\begin{ntn}
  \label{ntn:Connectives}
  \begin{enumerate}
  \item
    We shall follow the convention that $\dotminus$, like $-$, binds from left to right, and define $\psi \dotminus n\varphi$ by induction on $n$:
    \begin{align*}
      & \psi \dotminus 0\varphi  = \psi,
      && \psi \dotminus (n+1)\varphi = \psi \dotminus n\varphi \dotminus \varphi = (\psi \dotminus n\varphi) \dotminus \varphi.
    \end{align*}
  \item We use $1$ as abbreviation for $\neg(\varphi_0 \dotminus \varphi_0)$, where $\varphi_0$ is any formula.
  \item We also define
    \begin{align*}
      & \varphi \wedge \psi = \varphi \dotminus (\varphi \dotminus \psi), && \varphi \vee \psi = \neg( \neg \varphi \wedge \neg \psi),
    \end{align*}
    observing that $v(\varphi \wedge \psi) = \min v(\varphi), v(\psi)$ and $v(\varphi \vee \psi) = \max v(\varphi), v(\psi)$ for all $v$.
  \end{enumerate}
\end{ntn}

\begin{rmk}
  \label{rmk:InfImpl}
  Logical implication in Łukasiewicz logic can be infinitary by nature.
  Indeed, let $\varphi_n = 1 \dotminus 2(1 \dotminus nP)$.
  Then $\varphi_n = 0$ if and only if $1 \dotminus nP \geq \half$, i.e., if and only if $P \leq \frac{1}{2n}$.
  Letting $\Sigma = \{\varphi_n\}_{n \in \bN}$ we have $\Sigma \vDash P$ even though there is no finite $\Sigma_0 \subseteq \Sigma$ such that $\Sigma_0 \vDash P$.
  % This means that finiteness assumption in \fref{prp:FinImpl} cannot be discarded.
\end{rmk}

Let $\cS$ be a Łukasiewicz logic generated by $\{P_i\colon i \in I\}$, and $\varphi \in \cS$.
Then the truth assignments to $\cS$ are in bijection with $[0,1]^I$, and every formula $\varphi \in \cS$ can be identified with a continuous function $\hat \varphi\colon [0,1]^I \to [0,1]$ by $\hat \varphi(v) = v(\varphi)$.

\begin{lem}
  \label{lem:LukCpct}
  Let $\cS$ be a Łukasiewicz logic, and assume that $\Sigma \subseteq \cS$ has no model.
  Then there are $n,m \in \bN$ and $\varphi_i \in \Sigma$ for $i < m$ such that $\vDash 1 \dotminus n\varphi_0 \dotminus \ldots \dotminus n\varphi_{m-1}$.
\end{lem}
\begin{proof}
  For every $n,m \in \bN$ and $\bar \varphi \in \Sigma^m$, let $\psi_{n,\bar \varphi} =
  1 \dotminus n\varphi_0 \dotminus \ldots \dotminus n\varphi_{m-1}$, and assume that $\not\vDash \psi_{n,\bar \varphi}$ for all $n$ and $\bar \varphi$.
  In particular, for all $n,m$ and $\bar \varphi \in \Sigma^m$ there is $v$ such that $v(\psi_{2n,\bar \varphi}) > 0$, whereby $\sum v(\varphi_i) < \frac{1}{2n}$ and thus $v(\psi_{n,\bar \varphi}) > \half$.
  Call this $v_{n,\bar \varphi}$.
  Note that if $n \leq n'$ and $\bar \varphi \subseteq \bar \varphi'$ then $v_{n',\bar \varphi'}(\psi_{n,\bar \varphi}) > \half$ as well.
  Since $[0,1]^I$ is compact, we obtain an accumulation point $v \in [0,1]^I$ such that $v(\psi_{n,\bar \varphi}) \geq \half$ for all $n,m \in \bN$ and $\bar \varphi \in \Sigma^m$.
  It follows that $v(\varphi) = 0$ for all $\varphi \in \Sigma$.
\end{proof}

The proof of \fref{lem:LukCpct} only uses the presence of the connectives $\dotminus$ and $\neg$ (the latter in order to obtain $1$) in the language, and the fact that the evaluation $\hat \varphi\colon v \mapsto v(\varphi)$ is continuous for all $\varphi$.
Thus, if we allowed additional continuous functions $f\colon [0,1]^n \to [0,1]$ as connectives in \fref{dfn:LukasiewiczLogic}, e.g., the unary connective $\half\colon x \mapsto \half[x]$, the same proof would hold.

Let us now consider formal deductions in Łukasiewicz logic.
Recall that by \fref{ntn:Connectives}, $\varphi\wedge\psi$ is abbreviation for $\varphi \dotminus (\varphi \dotminus \psi)$ (which would be $A\psi\varphi$ in the notation of \cite{Rose-Rosser:FragmentsManyValuedCalculi}).
Thus, the four axiom schemes which, according to \cite{Rose-Rosser:FragmentsManyValuedCalculi,Chang:ProofOfLukasiewiczAxiom}, form a complete deduction system, are:
\begin{align*}
  \tag{A1} & \varphi \dotminus \psi \dotminus \varphi \\
  \tag{A2} & (\rho \dotminus \varphi) \dotminus (\rho \dotminus \psi) \dotminus (\psi \dotminus \varphi) \\
  \tag{A3} & (\varphi\wedge\psi)\dotminus(\psi\wedge\varphi) \\
  \tag{A4} & (\varphi \dotminus \psi) \dotminus (\neg\psi \dotminus \neg\varphi)
\end{align*}

While Modus Ponens becomes:
\begin{gather*}
  \tag{MP} \frac{\varphi, \psi \dotminus \varphi}{\psi}
\end{gather*}

A \emph{deduction sequence} from a set of premises $\Sigma$ in this deduction system is a sequence of formulae, each of which is either a premise (i.e., a member of $\Sigma$), an axiom (i.e., an instance of A1-4, where $\varphi$, $\psi$ and $\rho$ can be any formulae), or is deduced by Modus Ponens from two earlier formulae in the sequence.
We say that $\varphi$ is \emph{deducible} from $\Sigma$, in symbols  $\Sigma \vdash \varphi$ (or $\Sigma \vdash_\cS \varphi$ if we wish to be explicit) if there exists a deduction sequence from $\Sigma$ containing $\varphi$.
Soundness of this deduction system (i.e., $\Sigma \vdash \varphi \Longrightarrow \Sigma \vDash \varphi$) is easy to verify.
A subset $\Sigma \subseteq \cS$ is \emph{contradictory} if $\Sigma \vdash \varphi$ for all $\varphi \in \cS$.
Otherwise it is \emph{consistent}.
The completeness result we referred to can be now stated as:
\begin{fct}[\cite{Rose-Rosser:FragmentsManyValuedCalculi,Chang:NewCompletenessProof}]
  \label{fct:LukCmpl}
  Let $\cS$ be a Łukasiewicz logic, and $\varphi \in \cS$.
  Then $\vDash \varphi$ if and only if $\vdash \varphi$.
\end{fct}

% Using previous results this extends to finite sets of premises:
% \begin{prp}
%   \label{prp:LukCmplFinPrem}
%   Let $\cS$ be a Łukasiewicz logic, $\varphi \in \cS$ and $\Sigma \subseteq \cS$ finite.
%   Then $\Sigma \vDash \varphi$ if and only if $\Sigma \vdash \varphi$.
% \end{prp}
% \begin{proof}
%   Right to left is by soundness (and does not require $\Sigma$ to be finite).
%   For left to right, let $\Sigma = \{\psi_i\colon i < n\}$.
%   Assuming $\Sigma \vDash \varphi$, by \fref{prp:FinImpl} there exists $m$ such that $\vDash \varphi \dotminus m\psi_0 \dotminus \ldots \dotminus m\psi_{n-1}$.
%   By \fref{fct:LukCmpl}: $\vdash  \varphi \dotminus m\psi_0 \dotminus \ldots \dotminus m\psi_{n-1}$.
%   It follows that $\Sigma \vdash \varphi$.
% \end{proof}

% On the other hand, by \fref{rmk:InfImpl} and finiteness of formal deduction, \fref{prp:LukCmplFinPrem} may fail for infinite sets of premises.
% We can still prove a weaker statement:

\begin{prp}
  \label{prp:LukCmplPrem}
  Let $\cS$ be a Łukasiewicz logic, and let $\Sigma \subseteq \cS$.
  Then $\Sigma$ is consistent if and only if it has a model.
\end{prp}
\begin{proof}
  One direction is by soundness.
  For the other, assume that $\Sigma$ has no model.
  Then by \fref{lem:LukCpct} there are $n$ and $\varphi_i \in \Sigma$ such that letting $\psi
  = 1 \dotminus n\varphi_0 \dotminus \ldots \dotminus n\varphi_{m-1}$ we have $\vDash \psi$.
  By \fref{fct:LukCmpl} we have $\vdash \psi$, and by Modus Ponens $\Sigma \vdash 1$.
  By \fref{fct:LukCmpl} we also have $\vdash \varphi \dotminus 1$ for every formula $\varphi$, so $\Sigma \vdash \varphi$ and $\Sigma$ is contradictory.
\end{proof}

Unfortunately, this is not quite what we need, and we shall require the following modifications:
\begin{enumerate}
\item We wish to allow non-free logics, i.e., logics which are not necessarily freely generated from a set of atomic propositions.
  In particular, such logics need not be well-founded (i.e., we may have an infinite sequence $\{\varphi_n\}_{n\in\bN}$ such that each $\varphi_{i+1}$ is a ``proper sub-formula'' of $\varphi_i$).
\item The set of connectives $\{\neg,\dotminus\}$ is not full in the sense of \cite{BenYaacov-Usvyatsov:CFO}.
  We should therefore like to introduce an additional unary connective, denoted $\half$, which consists of multiplying the truth value by one half.
\end{enumerate}

\begin{dfn}
  \label{dfn:ContLogic}
  A \emph{continuous propositional logic} is a non empty structure $(\cS,\neg,\half,\dotminus)$, where $\dotminus$ is a binary function symbol and $\neg,\half$ are unary function symbols.

  A \emph{homomorphism} of continuous propositional logics is a map which respects $\neg$, $\half$ and $\dotminus$.

  A \emph{truth assignment} to a continuous propositional logic $\cS$ is a homomorphism $v\colon \cS \to [0,1]$, where $[0,1]$ is equipped with the natural interpretation of the connectives.
  Models and logical entailment are defined in the same manner as above.

  We say that $\cS$ is \emph{free (over $\cS_0$)} if there exists a subset $\cS_0 \subseteq \cS$ such that $\cS$ if freely generated from $\cS_0$ by the connectives $\{\neg,\half,\dotminus\}$.
  In that case every map $v_0\colon \cS_0 \to [0,1]$ extends to a unique truth assignment.
\end{dfn}

The new connective $\half$ requires two more axioms:
\begin{align*}
  \tag{A5} & \half\varphi \dotminus (\varphi \dotminus \half\varphi) \\
  \tag{A6} & (\varphi \dotminus \half\varphi) \dotminus \half\varphi
\end{align*}
Formal deductions in the sense of continuous propositional logic are defined as earlier, allowing A1-6 as logical axiom schemes.

\begin{lem}
  \label{lem:Proof}
  For every continuous propositional logic $\cS$ (not necessarily free), $\varphi,\psi \in \cS$, $\Sigma \subseteq \cS$ and $n \in \bN$:
  \begin{enumerate}
  \item $\vdash \varphi \dotminus \varphi$.
  \item $\vdash (\varphi\dotminus  \psi) \dotminus
    (1 \dotminus n(\psi \dotminus \varphi))$.
  \item If $\Sigma,\varphi\dotminus \psi$ is contradictory then $\Sigma \vdash \psi \dotminus \varphi$.
  \end{enumerate}
\end{lem}
\begin{proof}
  \begin{enumerate}
  \item In Łukasiewicz logic we have $\vDash P \dotminus P$, and by \fref{fct:LukCmpl}, $\vdash P \dotminus P$.
    By substitution of $\varphi$ for $P$ we get a deduction for $\varphi \dotminus \varphi$ in $\cS$.
  \item Same argument.
  \item If $\Sigma,\varphi \dotminus \psi$ is contradictory then it is has no model.
    By the proof of \fref{prp:LukCmplPrem} there is $n \in \bN$ such that $\Sigma \vdash 1 \dotminus n(\varphi \dotminus \psi)$.
    Therefore $\Sigma \vdash \psi \dotminus \varphi$.
    \qedhere
  \end{enumerate}
\end{proof}

\begin{thm}
  \label{thm:ContCmpl}
  Let $\cS$ be a continuous propositional logic, not necessarily free, and let $\Sigma \subseteq \cS$.
  Then $\Sigma$ is consistent if and only if it is satisfiable.
\end{thm}
\begin{proof}
  Let $\cS^f$ be the Łukasiewicz logic freely generated by
  $\{P_\varphi\colon \varphi \in \cS\}$, and let:
  \begin{align*}
    \Sigma_0^f =
    & \{P_{\neg\varphi} \dotminus \neg P_\varphi, \neg P_\varphi\dotminus P_{\neg\varphi}\colon \varphi \in \cS\} \\
    & \cup \{ P_{\varphi\dotminus\psi} \dotminus (P_\varphi\dotminus P_\psi), (P_\varphi\dotminus P_\psi)\dotminus P_{\varphi\dotminus\psi} \colon \varphi,\psi \in \cS \} \\
    & \cup \{ P_{\half\varphi} \dotminus P_{\varphi \dotminus \half\varphi}, P_{\varphi \dotminus \half\varphi} \dotminus P_{\half\varphi} \colon \varphi \in \cS \} \\
    \Sigma^f = & \{P_\varphi\colon \varphi \in \Sigma\} \cup \Sigma_0^f.
  \end{align*}

  Assume that $\Sigma^f$ has a model $v^f$.
  Define $v\colon \cS \to [0,1]$ by $v(\varphi) = v^f(P_\varphi)$.
  Since $v^f \vDash \Sigma_0^f$, $v$ is a truth assignment in the sense of $\cS$, and is clearly a model of $\Sigma$.

  Thus, if $\Sigma$ has no model, neither does $\Sigma^f$.
  By \fref{prp:LukCmplPrem} $\Sigma^f$ is contradictory.
  Thus, for every $\psi \in \cS$ we have $\Sigma^f \vdash P_\psi$.
  Take any deduction sequence witnessing this, replacing every atomic proposition $P_\varphi$ with $\varphi$.
  If a formula was obtained from previous ones using Modus Ponens, the same holds after this translation.
  Premises from $\Sigma^f$ become translated to one of several cases:
  \begin{enumerate}
  \item Premises of the form $P_\varphi$ for $\varphi \in \Sigma$ are replaced with $\varphi \in \Sigma$.
  \item Premises of the first two kinds from $\Sigma_0^f$ are replaced with something of the form $\varphi \dotminus \varphi$, which we know is deducible without premises.
  \item Premises of the last kind from $\Sigma_0^f$ are translated to instances of the axioms schemes A5-6.
  \end{enumerate}
  We conclude that $\Sigma \vdash \psi$ for all $\psi \in \cS$, and $\Sigma$ is contradictory.
  The other direction is by easy soundness.
\end{proof}

% As we said earlier \fref{thm:ContCmplFinPrem} fails with infinite sets of premises or non-free logics.
% We can prove the full version if we are willing to weaken a little the conclusion that $\Sigma \vdash \varphi$.
Let $2^{-n}$ be abbreviation for $\half\cdots\half 1$ ($n$ times), where $1$ is still as per \fref{ntn:Connectives}, so $v(2^{-n}) = 2^{-n}$ for any truth assignment $v$.

\begin{cor}
  \label{cor:ContCmpl}
  Let $\cS$ be a continuous propositional logic, not necessarily free, $\Sigma \subseteq \cS$ and $\varphi \in \cS$.
  Then $\Sigma \vDash \varphi$ if and only if $\Sigma \vdash \varphi \dotminus 2^{-n}$ for all $n$.
\end{cor}
\begin{proof}
  Right to left is clear, so assume that $\Sigma \vDash \varphi$.
  Then $\Sigma \cup \{2^{-n}\dotminus \varphi\}$ is non-satisfiable, and therefore contradictory by \fref{thm:ContCmpl}.
  By \fref{lem:Proof}: $\Sigma \vdash \varphi \dotminus 2^{-n}$.
\end{proof}

\begin{rmk}
  With some more effort, one can prove that if $\cS$ is free and $\Sigma$ is finite, then $\Sigma \vDash \varphi$ if and only if $\Sigma \vdash \varphi$.
  This can be shown to fail if we drop either additional hypothesis, and in any case will not be required for our present purposes.
\end{rmk}

These completeness results are extended to the full continuous first order logic in \cite{BenYaacov-Pedersen:CompleteDeduction}.
We conclude with a word regarding the semantics of continuous propositional logics.

\begin{dfn}
  Let $\cS$ be a continuous propositional logic.
  Its \emph{Stone space} is defined to be the set $\widetilde \cS = \Hom(\cS,[0,1])$, namely the space of truth assignments to $\cS$.
  We equip $\cS$ with the induced topology as a subset of $[0,1]^\cS$ (i.e., with the point-wise convergence topology).

  For each $\varphi \in \cS$ we define a function $\hat\varphi\colon \widetilde \cS \to [0,1]$ by $\hat \varphi(v) = v(\varphi)$.
\end{dfn}

\begin{prp}
  \label{prp:SemanticsCPL}
  Let $\cS$ be a continuous propositional logic, $\widetilde \cS$ its Stone space, and let $\theta_\cS$ denote the map $\varphi \mapsto \hat \varphi$.
  \begin{enumerate}
  \item The space $\widetilde \cS$ is compact and Hausdorff.
  \item $\theta_\cS \in \Hom\bigl( \cS,C(\widetilde \cS,[0,1]) \bigr)$.
    In particular, each $\hat \varphi$ is continuous.
  \item For $\varphi,\psi \in \cS$ we have $\theta_\cS(\varphi) = \theta_\cS(\psi)$ if and only if $\varphi \equiv \psi$.
  \item The image of $\theta_\cS$ is dense in the uniform convergence topology on $C(\widetilde\cS,[0,1])$.
  \end{enumerate}
  Moreover, the properties characterise the pair $(\widetilde \cS,\theta_\cS)$ up to a unique homeomorphism.
\end{prp}
\begin{proof}
  That the image is dense is a direct application of a variant of the
  Stone-Weierstrass theorem proved in
  \cite[Proposition~1.4]{BenYaacov-Usvyatsov:CFO}.
  The other properties are immediate from the construction.

  We are left with showing uniqueness.
  Indeed, assume that $X$ is a compact Hausdorff space and $\theta\colon \cS \to C(X,[0,1])$ satisfies all the properties above.
  Define $\zeta\colon X \to \widetilde \cS$ by $\zeta(x)(\varphi) = \theta(\varphi)(x)$.
  Thus $\zeta$ is the unique map satisfying $\theta_\cS(\varphi) \circ \zeta = \theta(\varphi)$, and we need to show that it is a homeomorphism.
  Continuity is immediate.
  The image of $\theta$ is dense in uniform convergence and therefore separates points, so $\zeta$ is injective.
  Since $X$ is compact and Hausdorff $\zeta$ must be a topological embedding.
  In order to see that $\zeta$ is surjective it will be enough to show that its image is dense.
  So let $U \subseteq \widetilde \cS$ be a non empty open set, which must contain a non empty set of the form $\{ v \in \widetilde \cS\colon f(v) > 0 \}$ for some $f \in C(\widetilde \cS,[0,1])$.
  For $n$ big enough there is $v_0 \in \widetilde \cS$ such that $f(v_0) > 2^{-n+1}$.
  By density find $\varphi_0 \in \cS$ such that $\|\hat \varphi_0 - f\|_\infty < 2^{-n}$.
  and let $\varphi = \varphi_0 - 2^{-n} \in \cS$.
  Then $\{ v \in \widetilde \cS\colon v(\varphi) > 0 \} \subseteq U$ and $v_0(\varphi) \neq 0$.
  Since $\varphi \not\equiv 0$ there is $x \in X$ such that $\zeta(x)(\varphi) = \theta(\varphi)(x) \neq 0$, i.e., $\zeta(x) \in U$.
  This concludes the proof.
\end{proof}

\section{The theory of $[0,1]$-valued random variables}
\label{sec:AxRV}

From this point and through the end of this paper, we switch to the setting of continuous first order logic.
This means that structures, formulae, theory and so on, unless explicitly qualified otherwise, should be understood in the sense of \cite{BenYaacov-Usvyatsov:CFO} (or \cite{BenYaacov-Berenstein-Henson-Usvyatsov:NewtonMS}).

Let $(\Omega,\sF,\mu)$ be a probability space.
In \cite{BenYaacov:SchroedingersCat} we considered such a space via its probability algebra $\bar \sF$, namely the Boolean algebra of events $\sF$ modulo null measure difference.
Equivalently, the probability algebra $\bar \sF$ can be viewed as the space of $\{0,1\}$-valued random variables (up to equality a.e.).
Here we shall consider a very similar object, namely the space of $[0,1]$-valued random variables.
This space will be denoted $L^1\bigl( (\Omega,\sF,\mu),[0,1] \bigr)$, or simply $L^1(\sF,[0,1])$, where we consider that the measure $\mu$ is implicitly part of the structure of $\sF$.
We equip this space with the natural interpretation of the connectives $\neg$, $\half$ and $\dotminus$ (e.g., $(X \dotminus Y)(\omega) = X(\omega) \dotminus Y(\omega)$), as well as with the $L^1$ distance $d(X,Y) = \bE[ |X-Y| ]$, for which it is a complete metric space.
It is thus naturally a structure in the random variable language
\begin{gather*}
  \cL_{RV} = \{0,\neg,\half,\dotminus\}.
\end{gather*}
Throughout, we shall use $1$ as an abbreviation for $\neg0$ and $E(x)$ as an abbreviation for $d(x,0)$.
The intended interpretation of $E(x)$ is the expectation.
Notice that by definition, if $\cM$ is any $\cL_{RV}$-structure and $a \in M$ then $a = 0 \Longleftrightarrow d(a,0) = 0 \Longleftrightarrow E(a) = 0$.

\subsection{The theories $RV$ and $ARV$}

We shall use the results of \fref{sec:Proofs} to give axioms for the theory of $[0,1]$-valued random variables equipped with the $L^1$ metric, in the language $\cL_{RV}$ given above.

The term algebra $\cT_{RV}$ of $\cL_{RV}$ is a free propositional continuous logic (freely generated by the variables of the language together with the symbol $0$).
% Thus, if $\tau(\bar x) \in \cT_{RV}$ then $\vdash_{\cT_{RV}} \tau \Longleftrightarrow {\vDash_{\cT_{RV}} \tau}$ by \fref{thm:ContCmplFinPrem}.
Similarly, forgetting everything but the operations $\neg$, $\half$ and $\dotminus$, any $\cL_{RV}$-structure $\cM$ is a (ordinarily, non-free) continuous propositional logic.
Translating proofs from $\cT_{RV}$ to $\cM$ we have $\vdash_{\cT_{RV}} \tau \Longrightarrow {\vdash_\cM \tau(\bar a)}$ for all $\bar a \in \cM$.

We define the theory $RV$ to consist of the following axioms.
In each axiom we quantify universally on all free variables.
Keep in mind also that $x \wedge y$ is abbreviation for $x \dotminus (x \dotminus y)$.

\begin{align*}
  \tag{RV1} & E(x) = E(x\dotminus y) + E(y\wedge x) \\
  \tag{RV2} & E(1) = 1 \\
  \tag{RV3} & d(x,y) = E(x\dotminus y)+E(y\dotminus x) \\
  \tag{RV4} & \tau = 0 &\text{whenever } \vdash_{\cT_{RV}} \tau \\
  \intertext{$ARV$ is defined by adding the following axiom:}
  \tag{ARV} & \inf_y \left(
    E(y\wedge\neg y) \vee \left| E(y\wedge x) - \half[E(x)] \right|
  \right) = 0.
\end{align*}

\begin{lem}
  \label{lem:ExpectationDifference}
  Let $\cM$ be a model of RV1.
  Then for every $a,b \in M$:
  $$E(a) - E(b) \leq E(a\dotminus b) \leq E(a).$$
  In particular $\cM$ respects Modus Ponens: if $b = 0$ and $a\dotminus b = 0$ then $a = 0$.
\end{lem}
\begin{proof}
  Axiom RV1 implies first that $E(a) \geq E(a\dotminus b)$.
  But then $E(b) \geq E(b\dotminus(b\dotminus a)) = E(b\wedge a)$ whereby $E(a) - E(b) \leq E(a) - E(b\wedge a) = E(a\dotminus b)$.
  Modus Ponens follows.
\end{proof}

Thus, modulo RV1, the axiom scheme RV4 is equivalent to the finite set:
\begin{align*}
  \tag{RV4.1} & (x \dotminus y) \dotminus x = 0 \\
  \tag{RV4.2} & ((x \dotminus z) \dotminus (x \dotminus y))
  \dotminus (y \dotminus z) = 0 \\
  \tag{RV4.3} & (x \wedge y) \dotminus (y \wedge x) = 0 \\
  \tag{RV4.4} & (x \dotminus y) \dotminus (\neg y \dotminus \neg x) = 0 \\
  \tag{RV4.5} & \half x \dotminus (x \dotminus \half x) = 0 \\
  \tag{RV4.6} & (x \dotminus \half x) \dotminus \half x = 0
\end{align*}
Furthermore, modulo RV1, RV3 and RV4.1-4, axioms RV4.5-6 are further equivalent to:
\begin{align*}
  \tag{RV5} & \half x = x \dotminus \half x
\end{align*}
Indeed, left to right is by RV3.
Axioms RV1 and RV4.1-4 imply that $x \dotminus x = 0$ (by \fref{fct:LukCmpl}), giving right to left.

The following is fairly obvious:
\begin{fct}
  \label{fct:AxRV}
  Let $(\Omega,\sF,\mu)$ be a probability space and let $\cM = L^1(\sF,[0,1])$.
  Then $\cM \vDash RV$, and if $(\Omega,\sF,\mu)$ is atomless then $\cM \vDash ARV$.
\end{fct}

We now aim to prove the converse of \fref{fct:AxRV}.
\begin{lem}
  \label{lem:ConsequencesRV}
  Let $\cM \vDash RV$, $a,b \in M$.
  Then:
  \begin{enumerate}
  \item
    $d(a,a \dotminus b) = E(a\wedge b) \leq E(b)$.
    In particular, $a \dotminus 0 = a$.
  \item
    $a \dotminus a = 0$.
    In particular, the meaning of $1$ here agrees with \fref{ntn:Connectives}(ii).
  \item
    $a \dotminus \half a = \half a$, $\half a \dotminus a = 0$ and $E(\half a) = \half E(a)$.
  \item
    Define by induction $2^0 = 1$ (i.e., $2^0 = \neg0$) and $2^{-(n+1)} = \half 2^{-n}$.
    Then for all $n \in \bN$: $E(2^{-n}) = 2^{-n}$.
  \item
    $a = 0
    \Longleftrightarrow {\vdash_\cM a}
    \Longleftrightarrow {\vDash_\cM a}$.
  \item \label{item:ConsequencesRVEquiv}
    $a = b \Longleftrightarrow a \equiv_\cM b$.
  \end{enumerate}
\end{lem}
\begin{proof}
  \begin{enumerate}
  \item From RV4 we have $(a \dotminus b) \dotminus a = 0$ and using RV3 we obtain $d(a,a\dotminus b) = E(a\wedge b)$.
    By RV1 $E(a\wedge b) \leq E(b)$.
    The rest follows.
  \item This was already observed earlier, using \fref{fct:LukCmpl}.
  \item That $\half a = a \dotminus \half a$ was observed above (RV5).
    It follows that $a \wedge \half a = \half a$, so $\half a \dotminus a = 0$ by RV1 (with $x = \half a$, $y = a$).
    Again by RV1 (now with $x = a$, $y = \half a$) we obtain $E(a) = 2E\bigl( \half a \bigr)$.
  \item Immediate from the previous item.
  \item Assume that $\vdash_\cM a$.
    Then by RV1 (which implies Modus Ponens) and RV4.1-6 we have $a = 0$.
    Thus $a = 0 \Longleftrightarrow {\vdash_\cM a}$.
    The implication $\vdash_\cM a \Longrightarrow {\vDash_\cM a}$ is by soundness.
    Finally assume that $\vDash_\cM a$.
    Then for all $n$ we have $\vdash_\cM a \dotminus 2^{-n}$, whereby $a \dotminus 2^{-n} = 0$.
    Thus $E(a) = E(a\wedge2^{-n}) \leq E(2^{-n})= 2^{-n}$, for arbitrary $n$.
    It follows that $E(a) = 0$, i.e., that $a = 0$.
  \item Assume that $a \equiv_\cM b$, i.e., that $\vDash_\cM a \dotminus b$ and $\vDash_\cM b \dotminus a$.
    Be the previous item $a \dotminus b = b \dotminus a = 0$ whereby $a = b$.
    \qedhere
  \end{enumerate}
\end{proof}

Let $\widetilde \cM$ be the Stone space of $\cM$, viewed as a continuous propositional logic, and let $\theta_\cM\colon \cM \to C(\widetilde \cM,[0,1])$ be as in \fref{prp:SemanticsCPL}.
Recall the notation $\hat a = \theta_\cM(a)$.
By \fref{lem:ConsequencesRV}\ref{item:ConsequencesRVEquiv} and \fref{prp:SemanticsCPL}, $\theta_\cM$ is injective.

The space $C(\widetilde \cM,[0,1])$ is naturally equipped with the supremum metric, denoted $\|f-g\|_\infty$.
We aim to show now that $d^\cM$ is an $L^1$ distance, i.e., that for an appropriate measure we have $d^\cM(a,b) = \|\hat a-\hat b\|_1$, which need not be equal to $\| \hat a - \hat b\|_\infty$.
Nonetheless, we can relate the two metrics as follows (we essentially say that $L^\infty$-convergence of random variables implies $L^1$-convergence).

\begin{lem}
  Assume that $\{a_n\}_{n\in\bN} \subseteq M$
  is such that
  $\{\hat a_n\}_{n\in\bN} \subseteq C(\widetilde \cM,[0,1])$
  is a Cauchy sequence in the supremum metric.
  Then $\{a_n\}_{n\in\bN}$ converges in $\cM$ and
  $\lim \hat a_n = \widehat{\lim a_n}$.
\end{lem}
\begin{proof}
  By assumption, for every $k < \omega$ there is $N_k$ such that for all
  $\|\hat a_n - \hat a_m\|_\infty \leq 2^{-k}$
  for all $n,m < N_k$.
  Therefore $(\hat a_n \dotminus \hat a_m) \dotminus 2^{-k} = 0$, and
  since $\theta_\cM$ is injective:
  $a_n \dotminus a_m \dotminus 2^{-k} = 0$.
  Thus
  $E(a_n \dotminus a_m) = E((a_n \dotminus a_m)\wedge2^{-k})
  \leq E(2^{-k}) = 2^{-k}$.
  Similarly $E(a_m \dotminus a_n) \leq 2^{-k}$, whereby
  $d(a_n,a_m) \leq 2^{-k+1}$.
  Since $\cM$ is a (complete) $\cL$-structure, it contains a limit $a$.

  Now fix $n \geq N_k$ and let $m \to \infty$.
  Then $a_m \to a$, and therefore
  $a_m \dotminus a_n \dotminus 2^{-k}
  \to a \dotminus a_n \dotminus 2^{-k}$.
  Thus
  $a \dotminus a_n \dotminus 2^{-k} = 0$, and by a similar argument
  $a_n \dotminus a \dotminus 2^{-k} = 0$.
  We have thus shown that
  $\hat a_n \to \hat a$ uniformly as desired.
\end{proof}

\begin{cor}
  The map $\theta_\cM\colon \cM \to C(\widetilde \cM,[0,1])$ is bijective.
\end{cor}
\begin{proof}
  We already know it is injective, and by
  \fref{prp:SemanticsCPL} its image is dense.
  By the previous lemma its image is complete, so it is onto.
\end{proof}

We shall identify $\cM$ with $C(\widetilde \cM,[0,1])$.
\begin{lem}
  \label{lem:LinearE}
  For all $a,b \in \cM$ and $r \in \bR^+$:
  \begin{enumerate}
  \item If $a+b \in \cM$ (i.e., $\|a+b\|_\infty \leq 1$) then
    $E(a+b) = E(a)+E(b)$.
  \item If $ra \in \cM$ (i.e., $r\|a\|_\infty \leq 1$) then $E(ra) = rE(a)$.
  \end{enumerate}
\end{lem}
\begin{proof}
  \begin{enumerate}
  \item Let $c = a+b$.
    Then $c \dotminus b = a$ and $b \dotminus c = 0$, whereby:
    \begin{align*}
      E(c) & = E(c \dotminus b) + E(b \dotminus (b \dotminus c))
      = E(a) + E(b\dotminus 0) = E(a) + E(b).
    \end{align*}
  \item For integer $r$ this follows from the previous item, and the rational case follows.
    If $r_n \to r$ then $r_na \to ra$ in $C(\widetilde \cM,[0,1])$ and \emph{a fortiori} in $\cM$, so the general case follows by continuity of $E$.
    \qedhere
  \end{enumerate}
\end{proof}

\begin{thm}
  \label{thm:RV}
  Let $\cM \vDash RV$, $\widetilde \cM$ its Stone space and
  $\theta_\cM\colon \cM \to C(\widetilde \cM,[0,1])$
  as in \fref{prp:SemanticsCPL}.
  \begin{enumerate}
  \item As a topological space, $\widetilde \cM$ is compact and Hausdorff.
  \item The map
    $\theta_\cM\colon \cM \to C(\widetilde \cM,[0,1])$
    is bijective and respects the operations $\neg$, $\half$ and
    $\dotminus$ (i.e., it is an isomorphism of continuous
    propositional logics).
  \item There exists a regular Borel probability measure
    $\mu$ on $\widetilde \cM$ such that the natural map
    $\rho_\mu\colon C(\widetilde \cM,[0,1]) \to L^1(\mu,[0,1])$
    is bijective as well, and the composition
    $\rho_\mu \circ \theta_\cM\colon \cM \to L^1(\mu,[0,1])$
    is an isomorphism of $L_{RV}$-structures.
  \end{enumerate}
  Moreover, these properties characterise
  $(\widetilde \cM,\mu,\theta_\cM)$ up to a unique measure preserving
  homeomorphism.
\end{thm}
\begin{proof}
  The first two properties are already known.
  By \fref{lem:LinearE} we can extend $E$ by linearity from $C(\widetilde \cM,[0,1])$ to $C(\widetilde \cM,\bR)$, yielding a positive linear functional.
  By the Riesz Representation Theorem \cite[Theorem~2.14]{Rudin:RealAndComplexAnalysis} there exists a unique regular Borel measure $\mu$ on $\widetilde \cM$ such that $E(f) = \int f\, d\mu$.
  Since $E(1) = 1$, $\mu$ is a probability measure.

  The map $\cM \to L^1(\mu,[0,1])$ is isometric and in particular injective.
  Its image is dense (continuous functions are always dense in the $L^1$ space of a regular Borel measure).
  Moreover, since $\cM$ is a complete metric space the image must be all of $L^1(\mu,[0,1])$, whence the last item.

  The uniqueness of $\widetilde \cM$ as a topological space verifying the first two properties follows from \fref{prp:SemanticsCPL} and \fref{lem:ConsequencesRV}.\ref{item:ConsequencesRVEquiv}.
  The Riesz Representation Theorem then yields the uniqueness of $\mu$.
\end{proof}

We may refer to $(\widetilde \cM,\mu)$ (viewed as a topological space equipped with a Borel measure) as the \emph{Stone space} of $\cM$ or say that $\cM$ is \emph{based on} $(\widetilde \cM,\mu)$.

\begin{cor}
  \label{cor:RV}
  Let $\cM$ be an $\cL_{RV}$-structure.
  Then:
  \begin{enumerate}
  \item The structure $\cM$ is a model of $RV$ if and only
    if it is isomorphic to some $L^1(\sF,[0,1])$.
  \item A structure of the form $L^1(\sF,[0,1])$ is a model of $ARV$
    if and only if $(\Omega,\sF,\mu)$ is an atomless probability space.
  \end{enumerate}
\end{cor}

\begin{cor}
  Let $\cM \vDash RV$ be based on $(\widetilde \cM,\mu)$.
  Then every Borel function
  $\widetilde \cM \to [0,1]$ is equal almost
  everywhere to a unique continuous function.
\end{cor}

\subsection{Interpreting random variables in events and vice versa}

In the previous section we attached to every probability space
$(\Omega,\sF,\mu)$ the space $L^1(\sF,[0,1])$
of $[0,1]$-valued random variables
and axiomatised the class of metric structures arising in this manner.
While we cannot quite recover the original space $\Omega$ from
$L^1(\sF,[0,1])$ we do consider that $L^1(\sF,[0,1])$ retains all the
pertinent information

An alternative approach to coding a probability space in a metric
structure goes through its probability algebra, namely the space of
$\{0,1\}$-valued random variables.
It can be constructed directly as the Boolean algebra quotient
$\bar \sF = \sF/\sF_0$ where $\sF_0$ is the ideal of null measure sets.
In addition to the Boolean algebra structure, it is equipped with the
induced measure function $\mu\colon \bar \sF \to [0,1]$ and the metric
$d(a,b) = \mu(a \triangle b)$ (in fact, the measure $\mu$ is superfluous and can
be recovered as $\mu(x) = d(x,0)$).
The metric is always complete, so a probability algebra is
a structure in the language
$\cL_{Pr} = \{0,1,\cap,\cup,\cdot^c,\mu\}$.

Let us define the theory $Pr$ to consist of the following axioms,
quantified universally:
\begin{align*}
  \tag{Bool} & \text{The theory of Boolean algebras: }
  (x\cap y)^c = x^c\cup y^c,\ldots \\
  \tag{Pr1} & \mu(1) = 1 \\
  \tag{Pr2} & \mu(x)+\mu(y)=\mu(x \cup y)+\mu(x \cap y) \\
  \tag{Pr3} & d(x,y)=\mu(x \triangle y).
  \intertext{The theory $APr$ (atomless probability algebras) consists of $PA_0$ along with:}
  \tag{APr} & \sup_x \inf_y \left| \mu(y\wedge x)-\frac{\mu(x)}{2} \right| = 0
\end{align*}

\begin{fct}
  The class of probability algebras is elementary, axiomatised
  by $Pr$.
  The class of atomless probability algebras is elementary as well,
  axiomatised by $APr$.

  Moreover, the theory $APr$ eliminates quantifiers
  (it is the model completion of $Pr$).
  It is  $\aleph_0$-categorical (there is a unique complete separable
  atomless probability algebra), and admits no compact model,
  whereby it is complete.
  It is $\aleph_0$-stable and its notion of independence
  coincides with probabilistic independence.
  All types over sets (in the real sort) are stationary.
\end{fct}
\begin{proof}
  Most of this is shown in \cite[Example~4.3]{BenYaacov-Usvyatsov:CFO}.
  The fact regarding stability and independence were shown in \cite{BenYaacov:SchroedingersCat} in the setting of compact abstract theories.
  The arguments carry nonetheless to models of $APr$ in continuous logic.
\end{proof}

We wish to show that these two ways of coding a probability space in a metric structure are equivalent.
Specifically we shall show that for any probability space $(\Omega,\sF,\mu)$, the probability algebra $\bar \sF$ and the space $L^1(\sF,[0,1])$ of $[0,1]$-valued random variables are (uniformly) interpretable in one another.

\begin{prp}
  \label{prp:ProbabilityAlgebraDefinable}
  Let $\cM$ be a model of $RV$, say $\cM = L^1(\sF,[0,1])$.
  Then the $\cL_{Pr}$-structure $\bar \sF$ is quantifier-free definable in
  $\cM$ in a manner which does not depend on $\Omega$.
  More precisely:
  \begin{enumerate}
  \item We may identify an event $a \in \bar \sF$ with its characteristic
    function $\mathbf{1}_a \in M$.
    This identifies $\bar \sF$ with
    the subset $L^1(\sF,\{0,1\}) \subseteq M$ consisting of all
    $\{0,1\}$-valued random variables over $(\Omega,\sF,\mu)$.
  \item Under the identification of the previous item,
    $\bar \sF$ is a quantifier-free
    definable subset of $\cM$, that is, the predicate
    $d(x,\bar \sF)$ is quantifier-free definable in $\cM$.
    Moreover, the Boolean algebra operations of $\bar \sF$ are definable by
    terms in $\cM$, and the predicates of $\bar \sF$
    (measure and distance) are quantifier-free definable in $\cM$.
  \end{enumerate}
  (See the first section of \cite{BenYaacov:DefinabilityOfGroups} for
  facts regarding definable sets in continuous logic.)
\end{prp}
\begin{proof}
  The first item is a standard fact.
  For the second item, let $g \in L^1(\sF,[0,1])$, and let
  $a = \{g\geq\half\}$ (i.e., $a = \{\omega\colon g(\omega) \geq \half\}$).
  Then:
  \begin{gather*}
    d(g,\bar \sF) = d(g,\mathbf{1}_a) = E(g\wedge\neg g).
  \end{gather*}
  Given $a,b \in \bar \sF$ we have
  $\mathbf{1}_{a^c} = \neg\mathbf{1}_a$ and
  $\mathbf{1}_{a \setminus b} = \mathbf{1}_a \dotminus \mathbf{1}_b$, from
  which the rest of the Boolean algebra structure can be recovered.
  In addition $d^{\bar \sF}(a,b) = d^\cM(\mathbf{1}_a,\mathbf{1}_b)$ and
  $\mu(a) = E(\mathbf{1}_a)$.
\end{proof}

Since $\bar \sF$ is (uniformly) definable we may quantify over it.
Thus, modulo the theory $RV$, axiom ARV can be written more elegantly
as:
\begin{align*}
  \tag{ARV$'$} &
  \inf_{y\in\bar \sF} \left| E(y\wedge x) - \half[E(x)] \right| = 0.
\end{align*}

The converse is a little more technical, since the interpretation of $L^1(\sF,[0,1])$ in the structure $\bar \sF$ will necessarily be in an imaginary sort.
A similar interpretation of the space of $[0,\infty]$-valued random variables in a \emph{hyper}imaginary sort has already been discussed in \cite[Section~3]{BenYaacov:SchroedingersCat}.
The result we prove here is a little stronger and easier to work with, using the notion of an imaginary sort in a metric structure, introduced in \cite[Section~5]{BenYaacov-Usvyatsov:CFO}.

Let $D = \{k/2^n\colon n \in \bN, 0 < k < 2^n\}$ denote the set of all dyadic fractions in ${]}0,1{[}$, $D' = D \cup \{0,1\}$.
For $r \in D'$, let $n(r)$ be the least $n$ such that $2^nr$ is an integer (so $n(0) = 0$, and for $r \neq 0$, $n = n(r)$ is unique such that $2^nr$ is an odd integer).
We shall now construct by induction on $n(r)$ a family of $\cL_{Pr}$-terms $(\tau_r)_{r\in D'}$ in a sequence of distinct variables $X = (x_r)_{r\in D}$.
We start with $\tau_0 = 1$ and $\tau_1 = 0$.
If $n(r) = m > 0$ then $n(r \pm 2^{-m}) < m$ and we define:
\begin{gather*}
  \tau_r =
  \bigl( x_r \cup \tau_{r-2^{-m}} \bigr) \cap \tau_{r+2^{-m}}.
\end{gather*}
We may write such a term as $\tau_r(X)$, where it is understood
that only a finite subset of $X$ appears in $\tau_r$.
Let $\bar \sF$ be a probability algebra.
Let $(a_s)_{s\in D} \subseteq \bar \sF$ be any sequence of events, and let $b_r = \tau_r(a_s)_{s\in D}$.
Then the sequence $(b_r)_{r\in D}$ is necessarily decreasing, and if the original sequence
$(a_s)_{s\in D}$ is decreasing then the two sequences coincide.

Let us also define:
\begin{align*}
  \varphi_n(y,X) & = \sum_{k<2^n} 2^{-n} \mu(y \cap \tau_{k/2^n}), & \varphi & = \lim_n \varphi_n.
\end{align*}
Since $0 \leq \varphi_n - \varphi_{n+1} \leq 2^{-n-1}$, the limit exists uniformly and $\varphi$ is an $\cL_{Pr}$-definable predicate.

\begin{prp}
  \label{prp:RandomVariablesInterpretable}
  Let $(\Omega,\sF,\mu)$ be a probability space, $\cM = L^1(\sF,[0,1])$.
  Let $\bar \sF_\varphi$ be the sort of canonical parameters for instances $\varphi(y,X)$ over $\bar \sF$.
  For each random variable $f\in M$, let $f_r = \{f \leq r\}$ for $r \in D$ and let $\tilde f \in \bar \sF_\varphi$ be the canonical parameter of $\varphi(y,f_r)_{r\in D}$.
  \begin{enumerate}
  \item For every event $c \in \bar \sF$: \qquad $\varphi(c,\tilde f) = \varphi(c,f_r)_{r\in D} = \int_c f$.
  \item 
    The map $f \mapsto \tilde f$ is a bijection between $M$ and $\bar \sF_\varphi$.
  \item 
    Identifying $M$ with $\bar \sF$ in this manner, the $\cL_{RV}$-structure on $\cM$ is definable in $\bar \sF$ in a manner which does not depend on $\Omega$.
  \end{enumerate}
  Moreover, if we compose this interpretation of $L^1(\sF,[0,1])$ in $\bar \sF$ with the definition of $\bar \sF$ in $L^1(\sF,[0,1])$ discussed in \fref{prp:ProbabilityAlgebraDefinable} above in either order, there is a definable bijection between the original structure and its interpreted copy in a manner which is uniform in $\Omega$.
\end{prp}
\begin{proof}
  For the first item, the sequence $(f_r)_{r\in D}$ is decreasing so $\tau_r(f_s)_{s\in D} = f_r$.
  It follows that $|\varphi_n(c,f_r)_{r\in D} - \int_c f| < 2^{-n}$ and $\varphi(x,f_r)_{r\in D} = \int_c f$.

  We now show the second item.
  To see that $f \mapsto \tilde f$ is injective assume that $\tilde f = \tilde g$.
  By the previous item this means that $\int_c f = \int_c g$ for every $c \in \bar \sF$, whereby $f=g$.
  To see it is surjective let $\varphi(x,A)$ be any instance of $\varphi$.
  Define:
  \begin{align*}
    b_r  & = \tau_r(A) \in \bar \sF & r \in D, \\
    f_n & = \sum_{k<2^n} 2^{-n} \mathbf{1}_{b_{k/2^n}} \in M & n \in \bN.
  \end{align*}
  One readily checks that $d(f_n,f_m) < 2^{-\min(n,m)}$, so the sequence $f_n$ converges to a limit $g \in M$ with $d(f_n,g) \leq 2^{-n}$.
  For every event $c \in \bar \sF$ we have $\varphi_n(c,A) = \int_c f_n$.
  It follows that $|\varphi_n(c,A) - \int_c g| \leq 2^{-n}$ and therefore $\varphi(c,A) = \int_c g$.
  In other words, $\tilde g$ is a canonical parameter for $\varphi(x,A)$.

  Let us now prove the third item.
  In order to prove that $(\tilde f,\tilde g) \mapsto \widetilde{f \dotminus g}$ is definable it is enough to show that we can define the predicate $\varphi\left( x,\widetilde{f \dotminus g} \right)$ uniformly from $\tilde f$ and $\tilde g$.
  Indeed:
  \begin{align*}
    \varphi\left( x,\widetilde{f \dotminus g} \right) &
    =
    \int_x (f \dotminus g)
    =
    \sup_y \, \left[
      \int_{x\cap y} f \dotminus \int_{x\cap y} g
    \right]
    \\ &
    =
    \sup_y \, \left[
      \varphi(x\cap y,\tilde f) \dotminus \varphi(x\cap y,\tilde g)
    \right].
  \end{align*}
  Similarly:
  \begin{gather*}
    \int_x 0 = 0,
    \qquad
    \int_x \neg f
    =
    \neg \int_x f,
    \qquad
    \int_x \half f
    =
    \half \int_x f.
  \end{gather*}
  It follows that all the connectives which one can construct from these primitives are definable, and in particular $(x,y) \mapsto |x-y|$.
  Thus the distance $d(f,g) = \varphi\left( 1,\widetilde{|f-g|} \right)$ is definable.

  We leave the moreover part to the reader.
\end{proof}

The intrinsic distance on the imaginary sort $\bar \sF_\varphi$ is by definition:
\begin{gather*}
  d_\varphi(f,g) = \sup_{b \in \bar \sF} \left|\int_b(f-g)\right|
  = \max\bigl( \|f \dotminus g\|_1, \|g \dotminus f\|_1 \bigr).
\end{gather*}
The distance $d_\varphi$ is easily verified to be uniformly equivalent to
the $L^1$ metric on the space of $[0,1]$-valued random variables.
This is a special case of the general fact that any two definable
distance functions on a sort are uniformly equivalent.
At the cost of additional technical complexity we could have arranged
to recover $L^1(\sF,[0,1])$ on an imaginary sort in which the intrinsic
distance is already the one coming from $L^1$.
Indeed, we could have defined a formula $\psi(y,z,x_r)_{r\in D}$ such that
$\psi(b,c,f_r)_{r\in D} = \int_b f + \int_{c \setminus b} \neg f$,
obtaining further down the road:
\begin{gather*}
  d_\psi(f,g)
  =
  \sup_{b,c \in \bar \sF} \left|\int_b(f-g) + \int_{c \setminus b} (g-f)\right|
  = \|f-g\|_1.
\end{gather*}

\subsection{Additional properties of $RV$ and $ARV$}

Models of $RV$ admits quantifier-free definable continuous functional calculus.
\begin{lem}
  \label{lem:ContinuousFunctionalCalculus}
  If $\theta\colon [0,1]^\ell \to [0,1]$ is a continuous function, then the function $\bar f \mapsto \theta \circ (\bar f)$ is uniformly quantifier-free definable in models of $RV$.
  By ``quantifier-free definable'' we mean that for every definable predicate $P(\bar y,z)$, the definable predicate $P(\bar y,\theta\circ(\bar x))$ is definable with the same quantifier complexity.
  Specifically, $d(y,\theta\circ(\bar x))$ is quantifier-free definable.
\end{lem}
\begin{proof}
  We can uniformly approximate $\theta$ by a sequence of terms $\tau_n(\bar x)$ in $\neg,\half,\dotminus$, in which case $P(\bar y,\theta\circ(\bar x)) = \lim P(\bar y,\tau_n(\bar x))$ uniformly.
\end{proof}
For example, the predicates $E(x^p)$ or $E( |x-y|^p )$ are definable for every $p \in [1,\infty{[}$, and thus the $L^p$ distance $\|x - y\|_p = E( |x-y|^p)^{1/p}$ is definable as well, all the definitions being quantifier-free and uniform.

For $A \subseteq L^1(\sF,[0,1])$, let $\sigma(A) \subseteq \sF$ denote the minimal
$\sigma$-sub-algebra by which every member of $A$ is measurable, i.e., such
that $A \subseteq L^1(\sigma(A),[0,1])$
(For this to be entirely well-defined we may require $\sigma(A)$ to contain
the null measure ideal of $\sF$.)
\begin{lem}
  \label{lem:SubModelRV}
  Let $\cM$ be a model of $RV$, say $\cM = L^1(\sF,[0,1])$.
  Then for every $\sigma$-sub-algebra $\sF_1 \subseteq \sF$, the space
  $L^1(\sF',[0,1])$ is a sub-structure of $\cM$.
  Conversely, every sub-structure $\cN \subseteq \cM$
  arises in this manner as $L^1(\sigma(N),[0,1])$.
\end{lem}
\begin{proof}
  The first assertion is clear, so we prove the converse.
  It is also clear that
  $N \subseteq L^1(\sigma(N),[0,1])$.

  Let $f \in N$, and define
  $\dot{m}f = f \dotplus \ldots \dotplus f$ ($m$ times).
  Then $\dot{m}f \in N$, and as $m \to \infty$ we have
  $\dot{m}f\to\mathbf{1}_{\{f>0\}}$ in
  $L^1$, so $\mathbf{1}_{\{f>0\}} \in N$.
  Since $N$ is complete and closed under $\neg$ and $\dotminus$,
  it follows that $\mathbf{1}_A \in N$ for every $A \in \sigma(N)$.
  Considering finite sums of the form
  $(\half)^k\mathbf{1}_{A_0} \dotplus \cdots
  \dotplus (\half)^k\mathbf{1}_{A_{n-1}}$ we see that every simple
  function in $L^1(\sigma(N),[0,1])$ whose range consists solely of dyadic
  fractions belongs to $N$.
  Using the completeness of $N$ one last time we may conclude that
  $L^1(\sigma(N),[0,1]) \subseteq  N$.
\end{proof}

\begin{lem}
  \label{lem:QFTypeJointDistribution}
  Let $\cM$ and $\cN$ be two models of $RV$, say
  $\cM = L^1(\sF,[0,1])$, $\cN = L^1(\Omega',[0,1])$.
  Then two $\ell$-tuples $\bar f \in M^\ell$ and $\bar g \in N^\ell$ have the same
  quantifier-free type in $\cL_{RV}$ if and only if they have the same
  joint distribution as random variables.
\end{lem}
\begin{proof}
  Assume that $\bar f \equiv^{qf} \bar g$.
  By the previous Lemma we have $E(\theta\circ(\bar f)) = E(\theta\circ(\bar g))$ for
  every continuous function $\theta\colon [0,1]^\ell \to [0,1]$, which is enough in
  order to conclude that $\bar f$ and $\bar g$ have the same joint
  distribution.
  Conversely, assume that $\bar f$ and $\bar g$ have the same
  joint distribution.
  Then $E(\tau(\bar f)) = E(\tau(\bar g))$ for every term $\tau(\bar x)$.
  It follows that $\bar f \equiv^{qf} \bar g$.
\end{proof}

Let $\bar \sF_a$ denote the set of atoms in $\bar \sF$,
which we may enumerate as $\{A_i\colon i\in I\}$.
Then $I$ is necessarily countable and
every $f \in L^1(\sF,[0,1])$ can be written uniquely as
$f_0 + \sum_{i\in I} \alpha_i\mathbf{1}_{A_i}$, where $f_0$ is over the atomless
part and $\alpha_i \in [0,1]$.

\begin{lem}
  \label{lem:DefinableAtoms}
  The set $\bar \sF_a\cup\{0\}$ is uniformly definable in $\bar \sF$.
  In   $L^1(\sF,[0,1])$, both the sets 
  $\bar \sF_a\cup\{0\}$ (i.e.,
  $\{\mathbf{1}_A\colon A\in\bar \sF_a\} \cup \{0\}$) and 
  $\{\alpha\mathbf{1}_A\colon \alpha\in [0,1], A\in\bar \sF_a\}$ are uniformly definable.
\end{lem}
\begin{proof}
  For the first assertion
  let $\varphi(x)$ be the $\cL_{Pr}$-formula
  $\sup_y \, \bigl( \mu(x\cap y)\wedge\mu(x\setminus y) \bigr)$.
  If $A$ is an atom or zero then clearly $\varphi(A) = 0$.
  If $A$ is an event which is not an atom then the nearest atom to $A$
  is the biggest atom in $A$ (or any of them if there are several of
  largest measure, or $0$ if $A$ contains no atoms).
  Let us construct a partition of $A$ into two events $A_1$ and $A_2$
  by assigning the atoms in $A$ (if any) sequentially to $A_1$ or to
  $A_2$, whichever has so far the lesser measure, and by splitting the
  atomless part of $A$ equally between $A_1$ and $A_2$.
  If $B \subseteq A$ is an atom of greatest measure (or zero if there are
  none) then $|\mu(A_1) - \mu(A_2)| \leq \mu(B)$ and:
  \begin{align*}
    \varphi(A) & \geq \mu(A_1)\wedge \mu(A_2) \geq \half \mu(A) - \half \mu(B)
    = \half \mu(A\setminus B) \\ &
    = \half d\bigl( A,\bar \sF_a\cup\{0\} \bigr).
  \end{align*}
  Thus $\bar \sF_a\cup\{0\}$ is definable.

  For the second assertion, $\bar \sF_a\cup\{0\}$ is relatively definable in
  $\bar \sF$ which is in turn definable in $L^1(\sF,[0,1])$, so
  $\bar \sF_a\cup\{0\}$ is definable in $L^1(\sF,[0,1])$.
  We may therefore quantify over $\bar \sF_a\cup\{0\}$, and define:
  \begin{gather*}
    \psi_n(x) = \inf_{A\in\bar \sF_a\cup\{0\}} \bigwedge_{k\leq2^n}
    d\left( x, \hbox{$\frac{k}{2^n}$}\mathbf{1}_A \right).
  \end{gather*}
  Then $\lim \psi_n$ defines the distance to the last set.
\end{proof}

If follows that for each $n$, the set of events which can be written
as the union of at most $n$ atoms is definable, as is the set of all
finite sums $\sum_{i<n} \alpha\mathbf{1}_{A_i}$ where each $A_i$ is an atom
(or zero).
These definitions cannot be uniform in $n$, though.
Indeed, an easy ultra-product argument shows that the set of all atomic
events (i.e., which are unions of atoms) cannot be definable or even
type-definable, and similarly for the set of all random
variables whose support is atomic.

The atoms of a probability space always belong to the algebraic closure of the empty set (to the definable closure if no other atom has the same measure).
They are therefore rather uninteresting from a model theoretic point of view, and we shall mostly consider atomless probability spaces.

\begin{thm}
  \label{thm:ARV}
  \begin{enumerate}
  \item The theory $ARV$ is complete and $\aleph_0$-categorical.
  \item The theory $ARV$ eliminates quantifiers.
  \item The universal part of $ARV$ is $RV$, and
    $ARV$ is the model completion of $RV$.
  \item If $\cM = L^1(\sF,[0,1]) \vDash ARV$
    and $A \subseteq M$ then
    $\dcl(A) = \acl(A) = L^1(\sigma(A),[0,1]) \subseteq M$.
  \item Two tuples $\bar f$ and $\bar g$ have the same type over a set
    $A$ (all in a model of $ARV$) if and only if they have the same
    joint conditional distribution over $\sigma(A)$.
  \item The theory $ARV$ is $\aleph_0$-stable,
    independence coinciding with probabilistic
    independence, i.e.: $A \ind_B C \Longleftrightarrow \sigma(A) \ind_{\sigma(B)} \sigma(C)$.
    Moreover, types over any sets in the home sort (i.e., not over
    imaginary elements) are stationary.
  \end{enumerate}
\end{thm}
\begin{proof}
  Categoricity and completeness of $ARV$ follow from the
  analogous properties for $APr$.

  Assume that $\bar f$ and $\bar g$ are two $\ell$-tuples in models of
  $ARV$, $\bar f =^{qf} \bar g$.
  By \fref{lem:QFTypeJointDistribution} they have the same joint
  distribution.
  For every $(r,i) \in D\times\ell$ define events
  $a_{f,i} = \{f_i \leq r\}$, $b_{r,i} = \{g_i \leq r\}$.
  The joint distribution implies that
  $(a_{r,i})_{(r,i) \in D \times \ell} \equiv^{qf} (b_{r,i})_{(r,i) \in D \times \ell}$,
  and since $APr$ eliminates quantifiers:
  $(a_{r,i})_{(r,i) \in D \times \ell} \equiv (b_{r,i})_{(r,i) \in D \times \ell}$.
  Under the interpretation of \fref{prp:RandomVariablesInterpretable}
  we have $f_i \in \dcl\bigl( (a_{r,i})_{r\in D} \bigr)$,
  $g_i \in \dcl\bigl( (b_{r,i})_{r\in D} \bigr)$, so
  $\bar f \equiv \bar g$.
  In other words, the quantifier-free type of a type determines its
  type, whence quantifier elimination.

  The theory $RV$ is universal and all its models embed in models of
  $ARV$, whereby $RV = ARV_\forall$.
  Since $ARV$ eliminates quantifiers it is the model completion of its
  universal part.

  Let now $\cM = L^1(\sF,[0,1]) \vDash ARV$, and let $A \subseteq M$.
  By \fref{lem:SubModelRV},
  $\langle A\rangle$ (the sub-structure generated by $A$ in $\cM$)
  is $L^1(\sigma(A),[0,1])$.
  Identifying $\bar \sF$ with its definable copy in $\cM$ we obtain
  $\langle A\rangle = \langle\sigma(A)\rangle= L^1(\sigma(A),[0,1])$ and
  $\dcl(A) = \dcl(\sigma(A)) \supseteq L^1(\sigma(A),[0,1])$.
  On the other hand, $\sigma(A)$ is a complete sub-algebra of
  $\bar \sF \vDash APr$ and therefore definably and even algebraically
  closed there.
  By our biïnterpretability result, $\sigma(A)$ is relatively
  algebraically closed in the definable copy of $\bar \sF$ in $\cM$.
  Therefore, if $g \in \acl(A) = \acl(\sigma(A))$ then $\sigma(g) \subseteq \sigma(A)$,
  i.e., $g \in L^1(\sigma(A),[0,1])$.

  Let us identify $\bar \sF$ with its definable copy in $\cM$,
  and let $\sA = \sigma(A)$.
  By the previous item we have
  $\tp(\bar f/A) \equiv \tp(\bar f/\sA)$.
  When $\sA = \{a_0,\ldots,a_{m-1}\}$ is finite sub-algebra,
  it is easy to verify that the joint
  conditional distribution of $\bar f$ over $\sA$ is the same as the
  joint distribution of the $(n+m)$-tuple $\bar f,\mathbf{1}_{\bar a}$.
  The result for types over infinite algebras follows.

  Stability and the characterisation of independence for $ARV$ follow
  from the analogous properties for $APr$
  via biïnterpretability.
\end{proof}

\section{Keisler randomisation}
\label{sec:Randomisation}

In this section we use earlier results to extends H.\ Jerome Keisler's
notion of a \emph{randomisation} of a classical
structure, or of a classical theory, to continuous ones.
For the original constructions we refer the reader to
\cite{Keisler:Randomizing,BenYaacov-Keisler:MetricRandom}.
Throughout, we work with a fixed continuous signature $\cL$.
For the sake of simplicity we shall assume that $\cL$ is
single-sorted, but everything we do can be carried out in a multi
sorted setting.

\subsection{Randomisation}

We shall want to consider some notion of probability integration of functions on a space $\Omega$, which is going to be additive, although not always $\sigma$-additive (i.e., possibly failing the Monotone Convergence Theorem and its consequences).
We do this by replacing the usual measure space apparatus with an abstract integration functional.

\begin{dfn}
  \label{dfn:IntegrationSpace}
  A \emph{finitely additive probability integration space}, or simply an \emph{integration space}, is a triplet $(\Omega,\sA,E)$ where $\Omega$ is any set, $\sA \subseteq [0,1]^\Omega$ is non empty and closed under the connectives $\neg$, $\half$ and $\dotminus$, and $E\colon \sA \to [0,1]$ satisfies $E(1) = 1$ and $E(X+Y) = E(X) + E(Y)$ whenever $X$, $Y$ and $X+Y$ are all in $\sA$.

  In this case we also say that $E$ is a \emph{finitely additive probability integration functional}, or simply an \emph{integration functional}, on $\sA$.
\end{dfn}

\begin{fct}
  Let $(\Omega,\sF,\mu)$ be a probability space.
  Let $\sA \subseteq [0,1]^\Omega$ consist of all $\sF$-measurable functions and let $E\colon \sA \to [0,1]$ be integration $d\mu$.
  Then $(\Omega,\sA,E)$ is an integration space.

  Similarly if $\sF$ is a mere Boolean algebra, $\mu$ is additive, and $\sA$ consists of all simple $\sF$-measurable functions.
\end{fct}

\begin{lem}
  Let $(\Omega,\sA,E)$ be an integration space.
  Equip $\sA$ with the distance $d(X,Y) = E(|X-Y|)$.
  Then $E(X) = d(X,0)$ for all $X \in \sA$ and $(\sA,0,\neg,\half,\dotminus,d)$ is a pre-model of $RV$.
\end{lem}
\begin{proof}
  Indeed, $d(X,0) = E(|X|) = E(X)$.
  Now RV1,2 follows from the hypothesis and the fact that $X = (X \dotminus Y) + (X \wedge Y)$.
  RV3 holds by definition.
  It follows from the hypothesis that $E(0) = 2E(0) = 0$, whence RV4.
\end{proof}

  In this situation we shall say that $(\sA,E)$ is a pre-model of $RV$, or that $E$ \emph{renders} $\sA$ a pre-model of $RV$.
If $E$ renders $\sA$ a pre-model of $ARV$ then we say that $(\Omega,\sA,E)$ is \emph{atomless}.

Let $\Omega$ be an arbitrary set and let $\sM = \sM_\Omega = \{\cM_\omega\}_{\omega\in\Omega}$ be a family of $\cL$-structures.
The product $\prod \sM = \prod_{\omega\in\Omega} M_\omega$ consists of all functions $\ra\colon \Omega \to \bigcup M_\omega$ which satisfy $\ra(\omega) \in M_\omega$ for all $\omega \in \Omega$.
Function symbols and terms of $\cL$ are interpreted naturally on $\prod \sM$.
For an $\cL$-formula $\varphi(\bar x)$ we define
\begin{gather*}
  \< \varphi(\bar \ra) \> \in \Omega^{[0,1]}, \qquad
  \< \varphi(\bar \ra) \>\colon \omega
  \mapsto \varphi^{\cM_\omega}(\bar \ra(\omega)).
\end{gather*}

\begin{dfn}
  \label{dfn:Randomisation}
  Let $\Omega$ be a set,
  $\sM_\Omega = \{\cM_\omega\}_{\omega \in \Omega}$
  a family of $\cL$-structures.
  Let also
  $\rM \subseteq \prod \sM$,
  $\sA \subseteq [0,1]^\Omega$ and
  $E\colon \sA \to [0,1]$.
  We say that $(\rM,\sA,E)$ is a
  \emph{randomisation based on $\sM_\Omega$} if
  \begin{enumerate}
  \item The triplet $(\Omega,\sA,E)$ is an integration space.
  \item The subset $\rM \subseteq \prod \sM$ is non empty, closed
    under function symbols, and
    $\< P(\bar \ra) \> \in \sA$
    for every $n$-ary predicate symbol
    $P \in \cL$ and $\bar \ra \in \rM^n$.
  \end{enumerate}
  We equip $\rM$ with the pseudo-metric
  $d( \ra,\rb ) = E\< d(\ra,\rb) \>$
  and $\sA$
  with the $L^1$ pseudo-metric
  $d(X,Y) = E\bigl( |X-Y| \bigr)$.

  We may choose to consider $E$ as part of the structure on
  $\sA$, in which case the randomisation is denoted by the pair
  $(\rM,\sA)$ alone.

  If $(\Omega,\sF,\mu)$ is a probability space,
  every $X \in \sA$ is $\sF$-measurable and
  $E[X] = \int X\,d\mu$ then we say that
  $(\rM,\sA)$ is based on the \emph{random family}
  $\sM_{(\Omega,\sF,\mu)}$ (and then we almost always omit $E$ from
  the notation).
\end{dfn}

The \emph{randomisation signature} $\cL^R$ is defined as follows:
\begin{itemize}
\item The sorts of $\cL^R$ include the sorts of $\cL$, referred to as \emph{main sorts}, plus a new \emph{auxiliary sort}.
\item Every function symbol of $\cL$ is present in $\cL^R$, between the corresponding main sorts.
  It is equipped with the same uniform continuity moduli as in $\cL$.
\item For every predicate symbol $P$ of $\cL$, $\cL^R$ contains a function symbol $\[ P \]$ from the corresponding main sorts into the auxiliary sort.
  It is equipped with the same uniform continuity moduli as $P$ in $\cL$.
\item The auxiliary sort is equipped with the signature $\cL_{RV}$.
\end{itemize}
A randomisation $(\rM,\sA)$ admits a natural interpretation as an $\cL^R$-pre-structure $(\rcM,\sA)$.
The corresponding structure will be denoted $(\widehat \rcM, \widehat \sA)$, and the canonical map $[\cdot]\colon (\rcM,\sA) \to (\widehat \rcM,\widehat \sA)$.
We also say that the randomisation $(\rM,\sA)$ is a \emph{representation} of the structure $(\widehat \rcM, \widehat \sA)$.

\begin{exm}
  \label{exm:UltraProduct}
  A special case of a randomisation is when $\rM = \prod \sM$ (i.e., the set of all sections from $\Omega$ into $\sM_\Omega$), $\sA = [0,1]^\Omega$, $\cU$ is an ultra-filter on $\Omega$, and $E_\cU \colon [0,1]^\Omega \to [0,1]$ is the integration functional corresponding to limits modulo $\cU$, i.e., $E_\cU(X) = \lim_{\omega \to \cU} X(\omega)$.
  In this case $\widehat \sA = [0,1]$ and $\widehat \rcM$ can be identified with the ultra-product $\prod_\cU \sM$.
\end{exm}

\begin{dfn}
  We say that a randomisation $(\rM,\sA)$ is \emph{full} if for every $\ra,\rb \in \rM$ and $X \in \sA$, there is a function $\rc \in \rM$ satisfying:
  \begin{gather*}
    \rc(\omega) =
    \begin{cases}
      \ra(\omega) & X(\omega) = 1, \\
      \rb(\omega) & X(\omega) = 0, \\
      \text{anything} & \text{otherwise.}
    \end{cases}
  \end{gather*}
  We shall sometimes write $\bc = \langle X,\ra,\rb\rangle$, even though there is no uniqueness here.

  We say that $(\rM,\sA)$ is \emph{atomless} if $\sA$ is a pre-model of $ARV$ (i.e., if $(\Omega,\sA,E)$ is atomless).
\end{dfn}

\begin{exm}[Randomisation of a single structure]
  \label{exm:RandomisationSingleStructure}
  Let $\cM$ be a structure, $(\Omega,\sF,\mu)$ an atomless probability space.
  Let $\rM_c \subseteq M^\Omega$ consist of all functions $\ra\colon \Omega \to M$ which take at most countably many values in $M$, each on a measurable set.
  Define $\sA_c \subseteq [0,1]^\Omega$ similarly, equipping it with integration with respect to $\mu$.
  Then $(\rM_c,\sA_c)$ is a full atomless randomisation.

  Assume now that $(\Omega,\sF,\mu)$ is merely a finitely additive probability space, namely that $\sF$ is a mere Boolean algebra and $\mu$ is finitely additive.
  Let $\rM_f \subseteq M^\Omega$ and $\sA_f \subseteq [0,1]^\Omega$ consist of functions which take at most finitely many values, each on a measurable set.
  Again, $(\rM_f,\sA_f)$ is an atomless, full randomisation.

  If $(\Omega,\sF,\mu)$ is a true (i.e., $\sigma$-additive) probability space then both constructions are possible and $(\rcM_f,\sA_f) \subseteq (\rcM_c,\sA_c)$.
  It is not difficult to check that they have the same completion $(\widehat \rcM_f,\widehat \sA_f) =
  (\widehat \rcM_c,\widehat \sA_c)$.
  In particular, $\widehat \sA_f =\widehat \sA_c =
  L^1(\sF,[0,1])$.

  Moreover, the resulting structure only depends on $\sA = L^1(\sF,[0,1])$, and we denote it by $(\cM^\sA,\sA)$ (or just $\cM^\sA$).
\end{exm}

\subsection{The randomisation theory}

Our first task is to axiomatise the class of $\cL^R$-structures which
can be obtained from full atomless randomisations (and in particular
show that it is elementary).
We shall use $x,y,\ldots$ to denote variables of $\cL$,
$\rx,\ry,\ldots$ to denote the corresponding variables in the main sort
of $\cL^R$ and
$U,V,\ldots$ to denote variables in the auxiliary sort of $\cL^R$.
For simplicity of notation, an $\cL^R$-structure
$(\rcM,\sA)$ may be denoted by $\rcM$ alone.
In this case, the auxiliary sort will be denoted by $\sA^\rcM$ and we
may write somewhat informally $\rcM = (\rcM,\sA^\rcM)$.
When $\sA^\rcM \vDash RV$ we shall refer to the probability algebra
of $\sA^\rcM$ as $\sF^\rcM$ (so $\sA^\rcM  = L^1(\sF^\rcM,[0,1])$).

The ``base theory'' for randomisation, which will be denoted by
$T_0^{Ra}$, consists of the theory $RV$ for the auxiliary sort along
with the following additional axioms (we recall that $a \dotminus b \dotminus c = (a \dotminus b) \dotminus c$):
\begin{align*}
  \tag{R1$_f$} &
  \bigl( \delta_{f,i}(\varepsilon) \dotminus \[ d(\rx,\ry)\] \bigr)
  \wedge
  \bigl( \[ d(f(\bar \rx',\rx,\bar \ry'), f(\bar \rx',\ry,\bar \ry'))\]
  \dotminus \varepsilon \bigr) = 0 \\
  \tag{R1$_P$} &
  \bigl( \delta_{P,i}(\varepsilon) \dotminus \[ d(\rx,\ry)\] \bigr)
  \wedge
  \bigl( \[ P(\bar \rx',\rx,\bar \ry')\] \dotminus \[ P(\bar \rx',\ry,\bar \ry')\]
  \dotminus \varepsilon \bigr) = 0 \\
  \tag{R2} &
  d(\rx,\ry) = E\[ d(\rx,\ry) \]
  \\
  \tag{R3} &
  \sup_{U\in\sF} \inf_{\rz} \,
  E\Bigl[
  \bigl( \[ d(\rx,\rz)\]\wedge U \bigr) \vee \bigl( \[ d(\ry,\rz)\]\wedge\neg U \bigr)
  \Bigr]
\end{align*}
In axiom R1, $\delta_{s,i}$ denotes the uniform continuity modulus of the
symbol $s$ with respect to its $i$th argument, with
$|\bar \rx'| = i$ and $|\bar \ry'| = n_s -i -1$.
In axiom R3, $\sF$ denotes the probability algebra of the auxiliary
sort, over which, modulo $RV$, we may quantify.

The role of axiom R1 is to ensure that the values of
$\[ P(\bar \ra)\](\omega)$, $f(\bar \ra)(\omega)$
only depends on $\bar \ra(\omega)$ and respect the uniform continuity
moduli prescribed by $\cL$.
Axiom R2 is straightforward, requiring the distance in the main sort
of be the expectation of the random variable associated to
$\cL$-distance.
Axiom R3 is a gluing property, corresponding to fullness of
a randomisation.
It can be informally stated as
\begin{align*}
  \tag{R3'} &
  (\forall \rx\ry)(\forall U\in\sF)(\exists \rz) \Bigl(
  \[ d(\rx,\rz)\]\wedge U = \[ d(\ry,\rz)\]\wedge\neg U = 0
  \Bigr),
\end{align*}
where the existential
quantifier is understood to hold in the approximate sense.
We prove in \fref{lem:PreciseR3} below that it actually holds in the
precise sense.

\begin{lem}
  \label{lem:RandomisationT0R}
  Let $(\rM,\sA)$ be a randomisation.
  Then $(\rcM,\sA)$ is a pre-model of $RV$ (in the auxiliary sort) and of {\rm R1,2}.
  If $(\rM,\sA)$ is full then $(\rcM,\sA)$ is a pre-model of $T_0^{Ra}$.
\end{lem}
\begin{proof}
  All we have to show is that if $(\rM,\sA)$ is full then $(\rcM,\sA)$ verifies R3, or equivalently, $(\widehat \rcM,\widehat \sA)$ does.
  However, we chose to write R3 using a quantifier over a definable set, a construct which need not have the apparent semantics in a pre-structure such as $(\rcM,\sA)$, and we find ourselves forced to work with $(\widehat \rcM,\widehat \sA)$.
  (Indeed, since $\sA$ is a mere \emph{pre}-model of $RV$, the algebra of characteristic functions in $\sA$ may well be trivial.)

  Let $\widehat \sF$ denote the probability algebra of $\widehat \sA$ and let $A \in \widehat \sF$, $\ra, \rb \in \widehat \rcM$.
  First, choose $X \in \sA$ and $\ra',\rb' \in \rM$ such that $[\ra']$ and $[\rb']$ are very close to $\bone_A$, $\ra$ and $\rb$, respectively.
  Define (recalling that for $t \in [0,1]$ and $n \in \bN$, $\dot n t = (nt) \wedge 1$):
  \begin{align*}
    Y & = \dot 2(X \dotminus 1/4) \in \sA, \\
    \rc & = \langle Y,\ra',\rb'\rangle \in \rM && \text{(by fullness)}, \\
    W & = \bigl( \[ d(\ra',\rc)\]\wedge Y \bigr) \vee \bigl( \[ d(\rb',\rc)\]\wedge \neg Y \bigr) \in \sA.
  \end{align*}
  For every $\omega \in \Omega$ we have $Y(\omega) \in \{0,1\} \Longrightarrow W(\omega) = 0$, or in other words, $W(\omega) \neq 0 \Longrightarrow 0 < Y(\omega) < 1 \Longrightarrow 1/4 < X(\omega) < 3/4$.
  Thus $W \leq (\dot 4 X) \wedge (\dot 4 \neg X)$.
  Having chosen our approximations good enough (we allow ourselves to skip the detailed epsilon chase here), we see that $[W] \leq (\dot 4 [X]) \wedge (\dot 4 \neg [X])$ is arbitrarily close to $0$ and $[Y]$ close to $\bone_A$.
  We conclude that $\bigl( \[ d(\ra,[\rc])\]\wedge A \bigr) \vee \bigl( \[ d(\rb,[\rc])\]\wedge \neg A \bigr)$ can be arbitrarily close to $0$ in $\widehat \sA$, which is what we needed to prove.
\end{proof}

In order to prove a converse we need to construct, for every model
$\rcM \vDash T^{Ra}_0$, a corresponding randomisation.

\begin{dfn}
  Assume $(\rcM,\sA) \vDash T_0^{Ra}$.
  Let $(\Omega,\mu) = (\Omega^\rcM,\mu^\rcM)$ be the Stone space of $\sA$ as per \fref{thm:RV}.
  Then we say that $(\rcM,\sA)$ is \emph{based} on $(\Omega,\mu)$.
\end{dfn}

We recall that $\Omega$ is a compact Hausdorff topological space, $\mu$ is a regular Borel probability measure and we may identify $\sA = C(\Omega,[0,1]) = L^1(\mu,[0,1])$.
Under this identification $\int_\Omega X \,d\mu = E(X)$ for all $X \in \sA$.

For each $\omega \in \Omega$ we define an $\cL$-pre-structure $\cM_{0,\omega}$.
Its underlying set is $M_{0,\omega} = \rM$ and the interpretations of the symbols are inherited naturally from $\rcM$:
\begin{gather*}
  f^{\cM_{0,\omega}} = f^\rcM \colon \rM^n \to \rM,
  \qquad
  P^{\cM_{0,\omega}}(\bar \ra) = \[ P(\bar \ra)\](\omega) \in [0,1].
\end{gather*}
Notice that axiom R1$_d$ implies that
$\[ d(\rx,\ry)\]
\dotminus \[ d(\rx,\rz)\]
\leq \[ d(\ry,\rz)\]$ and
axiom R2 implies $\[ d(\rx,\rx)\] = 0$.
Symmetry of $\[ d(\rx,\ry)\]$ and the usual form of the triangle
inequality follow, so $d^{\cM_{0,\omega}}$ is a pseudo-metric for every
$\omega$.
Other instances of axiom R1 imply that $\cM_{0,\omega}$ respects
uniform continuity moduli prescribed by $\cL$.
Thus $\cM_{0,\omega}$ is indeed an $\cL$-pre-structure.
The structure associated to $\cM_{0,\omega}$ will be denoted
$\cM_\omega$.
Let $\sM$ denote the family $\{\cM_\omega\}_{\omega\in\Omega}$
and let $a_\omega$ denote the image of $\ra$ in $\cM_\omega$.

Assume that $\ra,\rb \in \rM$ are distinct.
Then $E\[ d(\ra,\rb)\] > 0$, whereby
$\[ d(\ra,\rb)\](\omega) > 0$
for some $\omega \in \Omega$.
Thus $a_\omega \neq b_\omega$ and the maps
$\omega \mapsto a_\omega$, $\omega \mapsto b_\omega$ are distinct.
In other words, we may identify $\ra \in \rM$ with
the map $\ra\colon \omega \mapsto a_\omega$.
Viewed in this manner we have $\rM \subseteq \prod \sM$.
By construction, if $f \in \cL$ is a function symbol then its
coordinate-wise action on $\rM$ as a subset of $\prod \sM$ coincides
with $f^\rcM$.
Similarly, if $P \in \cL$ is a predicate symbol then
$\< P(\bar \ra) \>
= \bigl( \omega \mapsto P^{\cM_\omega}(\bar \ra(\omega)) \bigr)
= \[ P(\bar \ra) \] \in \sA$.
We have thus identified $(\rM,\sA)$ with a randomisation
base on $(\Omega,\mu)$.
This randomisation is called the \emph{canonical representation} of
$(\rcM,\sA)$.

\begin{lem}
  \label{lem:PreciseR3}
  Let $\rcM \vDash T_0^{Ra}$, $\ra,\rb \in \rM$ and
  $A \in \sF^\rcM$.
  Then there exists (a unique)
  $\rc = \langle A,\ra,\rb\rangle \in \rM$ which is
  equal to $\ra$ over $A$ and to $\rb$ elsewhere:
  \begin{gather*}
    \[ d(\ra,\rc)\] \wedge A = \[ d(\rb,\rc)\] \wedge \neg A = 0.
  \end{gather*}
  Identifying $(\rcM,\sA^\rcM)$ with its canonical representation
  based on $\Omega$,
  $A$ is identified with a (unique) clopen set $A \subseteq \Omega$
  and we have:
  \begin{gather*}
    \rc(\omega) =
    \begin{cases}
      \ra(\omega) & \omega \in A, \\
      \rb(\omega) & \omega \notin A.
    \end{cases}
  \end{gather*}
\end{lem}
\begin{proof}
  By axiom R3, for every $\varepsilon > 0$, there is
  $\rc_\varepsilon$ such that:
  \begin{gather*}
    E\Bigl[
    \bigl( \[ d(\ra,\rc_\varepsilon)\]\wedge A \bigr)
    \vee
    \bigl( \[ d(\rb,\rc_\varepsilon)\]\wedge\neg A \bigr)
    \Bigr]
    < \varepsilon.
  \end{gather*}
  Passing to the canonical representation it is easy to check that
  $d(\rc_\varepsilon,\rc_{\varepsilon'}) < \varepsilon+\varepsilon'$
  for any $\varepsilon,\varepsilon' > 0$.
  Thus
  $(\rc_\varepsilon)_{\varepsilon\to 0^+}$ is a Cauchy
  sequence whose limit
  $\rc = \langle A,\ra,\rb\rangle$ is as desired.
  Uniqueness is clear.
\end{proof}

\begin{thm}
  \label{thm:ModelsT0R}
  An $\cL^R$-structure is a model of $T_0^{Ra}$ if and only if it has a
  full representation, i.e., if and only if it is
  isomorphic to a structure $(\widehat \rcM,\widehat \sA)$ associated
  to a full randomisation $(\rM,\sA)$.

  Moreover, let $(\rcM,\sA)$ be a model of $T_0^{Ra}$.
  Then the canonical representation of $(\rcM,\sA)$ is full,
  and as an $\cL^R$-pre-structure it is isomorphic to
  $(\rcM,\sA)$.
  In particular, the $\cL^R$-pre-structure associated to the
  canonical representation is already a structure.
\end{thm}
\begin{proof}
  One direction is \fref{lem:RandomisationT0R}, so it is enough to
  prove the moreover part.
  It is clear that the identity map is an isomorphism between the
  structure $(\rcM,\sA)$ and the pre-structure associated to the
  canonical representation, so all that is left to show is that the
  latter is full.

  Let $\ra,\rb \in \rM$, $X \in \sA$.
  The set $\{X \leq \half\} \subseteq \Omega$ is Borel
  and therefore equal outside a null measure set to
  some clopen set $U \subseteq \Omega$.
  We now have
  \begin{gather*}
    X \dotminus \mathbf{1}_{\{X\geq1/2\}} \leq \half
    \quad \Longrightarrow \quad
    X \dotminus \mathbf{1}_U \leq \half,
    \quad \Longrightarrow \quad
    U \supseteq \{X=1\},
    \\
    \mathbf{1}_{\{X\geq1/2\}} \dotminus X \leq \half
    \quad \Longrightarrow \quad
    \mathbf{1}_U \dotminus X \leq \half,
    \quad \Longrightarrow \quad
    U \cap \{X=0\} = \emptyset.
  \end{gather*}
  Thus $\rc = \langle U,\ra,\rb\rangle$ will do as
  $\langle X, \ra, \rb \rangle$.
\end{proof}

From now on we shall identify a model of $T_0^{Ra}$ with its
canonical representation whenever that is convenient and without
further mention.

\subsection{Quantifiers}

It is a classical fact that $\sA = L^1(\sF,[0,1])$ is a complete lattice.
More precisely, let $\cA \subseteq \sA$ be any subset.
We may assume that $\cA$ is closed under $\wedge$.
Let $r = \inf \{E(X)\colon X \in \cA\}$ and let $(X_n)_{n\in\bN} \subseteq \cA$
satisfy $E(X_n) \to r$.
By hypothesis $E(X_n \wedge X_m) \geq r$ whereby
$d(X_n,X_m) \leq |E(X_n)-r| + |E(X_m)-r|$.
The sequence $(X_n)_{n\in\bN}$ is therefore Cauchy and its limit is
$\inf \cA$.

Let now $(\rcM,\sA)$ be a model of $T_0^{Ra}$,

\begin{dfn}
  Let $(\rcM,\sA) \vDash T_0^{Ra}$, $t\colon \rM^n \to \sA$ a function.
  We say that $t$ is \emph{local} if it is always true that:
  \begin{gather*}
    t(\ldots,\langle A,\ra,\rb\rangle,\ldots) = 
    t(\ldots,\ra,\ldots) \wedge A + t(\ldots,\rb,\ldots) \wedge \neg A.
  \end{gather*}

  For a function $t\colon \rM^{n+1} \to \sA$
  we define $\sinf_\ry t(\bar \rx,\ry) \colon \rM^n \to \sA$ by
  \begin{gather*}
    \sinf_\ry t(\bar a,\ry)
    = \inf \{t(\bar \ra,\rb)\colon \rb \in \rM\} \in \sA.
  \end{gather*}
\end{dfn}

\begin{lem}
  \label{lem:RandomVarQuantifier}
  Let $t(\bar \rx,\ry)$ be a uniformly definable local function in models of
  $T_0^{Ra}$ from the main sort into the auxiliary sort.
  Then the function
  $s(\bar \rx) = \inf_{\ry} t(\bar \rx,\ry)$ is uniformly definable
  and local as well, and $T_0^{Ra}$ implies that:
  \begin{gather*}
    \sinf_\rz
    d\bigl( \sinf_\ry t(\bar \rx,\ry), t(\bar \rx,\rz) \bigr) = 0.
  \end{gather*}
  Moreover, for every $\bar \ra$ in a model of $T_0^{Ra}$ and $\varepsilon > 0$
  there is $\rb$ such that:
  \begin{gather*}
    t(\bar \ra,\rb) \leq \sinf_\ry t(\bar \ra,\ry) + \varepsilon
  \end{gather*}
  (Similarly for $\sup_\ry t$.)
\end{lem}
\begin{proof}
  It follows directly from the definition that if $t$ is local then so
  is $\inf_\ry t$ (no definability is needed here).

  We start by proving the moreover part.
  Let $(\rcM,\sA) \vDash T_0^{Ra}$, $\bar \ra \in \rM^n$.
  Following the discussion of the completeness of the lattice
  structure on $\sA$ there is a sequence
  $\{\rc_n\}_{n\in\bN}$ such that
  $\inf_\ry t(\bar \ra,\ry) = \inf_n t(\bar \ra,\rc_n)$.
  Let us define a sequence $\{\rb_n\}$ by:
  \begin{align*}
    & \rb_0 = \rc_0, &
    & \rb_{n+1} = \big\langle \{ t(\bar \ra,\rb_n) - t(\bar \ra,\rc_{n+1}) > \varepsilon \}, \rc_{n+1},\rb_n \big\rangle.
  \end{align*}
  In other words, when passing from $\rb_n$ to $\rb_{n+1}$ we use $\rc_{n+1}$ only where this means a decrease of more than $\varepsilon$, and elsewhere keep $\rb_n$.

  Clearly $\sum_n \mu\{ t(\bar \ra,\rb_n) - t(\bar \ra,\rc_{n+1}) > \varepsilon \} \leq 1/\varepsilon$.
  By construction, $d(\rb_n,\rb_{n+1}) \leq \mu\{ t(\bar \ra,\rb_n) - t(\bar \ra,\rc_{n+1}) > \varepsilon \}$, so the sequence $\{\rb_n\}$ converges to some $\rb$.
  Since $t$ is local, we have $t(\bar \ra,\rb) \leq t(\bar \ra,\rc_n) + \varepsilon$, whence $t(\bar \ra,\rb) \leq \inf_\ry t(\bar \ra,\ry) + \varepsilon$, as desired.

  We can now prove the first assertion.
  Indeed, it follows from the moreover part that the graph of $\inf_\ry t$ is uniformly definable as:
  \begin{gather*}
    X = \sinf_\ry t(\bar \ra,\ry)
    \quad \Longleftrightarrow \quad
    \begin{cases}
      \ssup_\rz E\bigl( X \dotminus t(\bar \ra,\rz) \bigr) = 0, \\
      \sinf_\rz E\bigl( t(\bar \ra,\rz) \dotminus X \bigr) = 0.
    \end{cases}
  \end{gather*}
  Once we know that $\inf_\ry f$ is definable, the sentence in the second assertion is expressible, and holds true by the moreover part.
\end{proof}

We now proceed to define by induction, for every $\cL$-formula $\varphi(\bar x)$, a $T_0^{Ra}$-definable local function $\[ \varphi(\bar \rx) \]$ to the auxiliary sort, in the following natural manner:
\begin{itemize}
\item \emph{Atomic formulae:}
  $\[ P(\bar \tau) \] = \[ P\] \circ (\bar \tau)$ is a term, the composition of the function symbol $\[ P \]$ with the $\cL$-terms $\bar \tau$, which are also $\cL^R$-terms.
  These are local by \fref{thm:ModelsT0R}.
\item \emph{Connectives:}
  $\[ \varphi \dotminus \psi \] = \[\varphi\] \dotminus \[\psi\]$, and so on.
  Locality is clear.
\item \emph{Quantifiers:}
  $\[ \inf_y \varphi(\bar \rx,y) \] = \inf_\ry \[ \varphi(\bar \rx,\ry) \]$, $\[ \sup_y \varphi(\bar \rx,y) \] = \sup_\ry \[ \varphi(\bar \rx,\ry) \]$.
  Locality follows from \fref{lem:RandomVarQuantifier}.
\end{itemize}

Our somewhat minimalist approach differs from that of Keisler, who introduces a function symbol $\[ \varphi(\bar \rx) \]$ for every $\cL$-formula $\varphi$ (see \cite{Keisler:Randomizing,BenYaacov-Keisler:MetricRandom}).
Keisler's \emph{Boolean Axioms} and \emph{Fullness Axiom} are valid in our setting by definition of $\[\varphi\]$ (using \fref{lem:RandomVarQuantifier} for
fullness).
Keisler's \emph{Distance Axiom} for the main sort is our R2.
While not entirely equivalent, Keisler's \emph{Event Axiom} corresponds to our axiom R3.
(More precisely, Keisler's Event Axiom is equivalent to R3 plus $\sup_{\rx,\ry} d(\rx,\ry) = 1$.
We do not find it necessary or desirable to assume the latter.)
Other axioms related to the auxiliary sort, with the exception of atomlessness, are coded in $RV$.
We shall add atomlessness later on, when it is needed for \fref{thm:TRTypes}.
We are left with the \emph{Validity Axioms} which we also claim follow from $T_0^{Ra}$.

\begin{thm}
  \label{thm:Satisfaction}
  Let $(\rcM,\sA)$ be a model of $T_0^{Ra}$ which we identify as usual with its canonical representation, based on $(\Omega,\mu)$.
  Then for every formula $\varphi(\bar x)$ and tuple $\bar \ra$ of the appropriate length we have $\< \varphi(\bar \ra) \> = \[ \varphi(\bar \ra) \]$ as functions on $\Omega$ (and not merely up to a null measure set).
\end{thm}
\begin{proof}
  We prove by induction on $\varphi$.
  If $\varphi$ is atomic this is known by construction and the
  induction step for connectives is immediate.
  We are left with the case of a formula $\inf_x \varphi(x,\bar y)$.
  First of all, by construction, we have:
  \begin{gather*}
    \< \sinf_x \varphi(x,\bar \ra) \>
    = \sinf^s \bigl\{ \< \varphi(\rb,\bar \ra) \> \colon \rb \in \rM \bigr\},
    \\
    \[ \sinf_x \varphi(x,\bar \ra) \]
    = \sinf_\rx \[ \varphi(\rx,\bar \ra) \]
    = \sinf^{L^1} \bigl\{ \[ \varphi(\rb,\bar \ra) \] \colon \rb \in \rM \bigr\}.
  \end{gather*}
  Here $\inf^s$ means the simple, or point-wise, infimum of functions
  on $\Omega$.
  By definition
  $\[ \inf_x \varphi(x,\bar \ra) \] \leq \[ \varphi(\rb,\bar \ra) \]$
  for all $\rb$,
  and by  the induction hypothesis for $\varphi$ we have
  $\[ \inf_x \varphi(x,\bar \ra) \] \leq \< \varphi(\rb,\bar \ra) \>$.
  It follows that
  $\[ \inf_x \varphi(x,\bar \ra) \] \leq
  \< \inf_x \varphi(x,\bar \ra) \>$.
  Conversely, by \fref{lem:RandomVarQuantifier},
  for every $\varepsilon > 0$ there exists $\rb$ such that
  $\[ \inf_x \varphi(x,\bar \ra) \] + \varepsilon \geq \[ \varphi(\rb,\bar \ra) \]$.
  Using the induction hypothesis again we obtain:
  \begin{gather*}
    \[ \sinf_x \varphi(x,\bar \ra) \] + \varepsilon
    \geq \< \varphi(\rb,\bar \ra) \>
    \geq \< \sinf_x \varphi(x,\bar \ra) \>.
  \end{gather*}
  Equality follows.
\end{proof}

\begin{cor}
  Let $\rcM \vDash T_0^{Ra}$ and assume its canonical representation
  is based on the family $\sM = \{\cM_\omega\}_{\omega\in\Omega}$.
  Then for every $\cL$-sentence $\varphi$:
  \begin{gather*}
    \rcM \vDash \[\varphi\]=0
    \quad \Longleftrightarrow \quad
    \cM_\omega \vDash \varphi \text{ for all } \omega \in \Omega.
  \end{gather*}
\end{cor}
\begin{proof}
  Immediate from the fact that
  $\[ \varphi \] = \< \varphi \>$ on $\Omega$.
\end{proof}

\begin{dfn}
  Let $T$ be a set of $\cL$-sentences.
  We define its randomisation $T^{Ra}$ to be the
  $\cL^R$-theory consisting of the base theory along with the
  translation of $T$ (Keisler's \emph{Transfer Axioms}):
  \begin{gather*}
    T^{Ra} = T_0^{Ra} \cup\{\[\varphi\] = 0\}_{\varphi\in T}.
  \end{gather*}
\end{dfn}

\begin{cor}
  Let $T$ be arbitrary set of sentences, $\varphi$ a sentence.
  Then $T \vdash \varphi \Longleftrightarrow T^{Ra} \vdash \[\varphi\] = 0$.
\end{cor}
\begin{proof}
  Immediate.
\end{proof}

\begin{cor}[Keisler's \emph{Validity Axiom}]
  Assume $\varphi$ is a valid $\cL$-sentence.
  Then $T_0^{Ra} \vdash \[\varphi\] = 0$.
\end{cor}

\subsection{A variant of Łoś's Theorem}

\begin{thm}[Łoś's Theorem for randomisation]
  \label{thm:RandomLos}
  Let $\sM_\Omega$ be a family of structures, $\rM = \prod \sM$, and let $E$ be an integration functional on $\sA = [0,1]^\Omega$.
  Let $(\widehat \rcM,\widehat \sA)$ denote the structure associated to the randomisation $(\rM,\sA)$.

  Then $(\rM,\sA)$ is full and for every formula $\varphi(\bar x)$ and every $\bar \ra \in \rM^n$:
    \begin{gather*}
      \bigl[ \< \varphi(\bar \ra) \> \bigr] =
      \[ \varphi([\bar \ra]) \].
    \end{gather*}
\end{thm}
\begin{proof}
  Fullness is immediate.
  We claim that
  $\bigl[ \< \sinf_y \varphi(\bar \ra,y) \> \bigr]
  = \inf_{\rb\in\rM}
  \bigl[ \< \varphi(\bar \ra,\rb) \> \bigr]$
  for every formula $\varphi(\bar x,y)$ and every
  $\bar \ra \in \rM^n$,
  where the infimum on the right hand side is in the sense of the
  lattice
  $\widehat \sA$.
  Indeed, the inequality $\leq$ is immediate.
  For $\geq$ observe
  that using the Axiom of Choice,
  for every $\varepsilon > 0$ we can find $\rb \in \rM$ such that
  $\< \sinf_y \varphi(\bar \ra,y) \> + \varepsilon \geq
  \< \varphi(\bar \ra,\rb) \>$,
  whereby
  $\bigl[ \< \sinf_y \varphi(\bar \ra,y) \> \bigr]
  + \varepsilon \geq
  \bigl[ \< \varphi(\bar \ra,\rb) \> \bigr]$.

  We now prove the main assertion.
  First of all, we may replace $\varphi$ with an equivalent formula
  $\psi$.
  Indeed, on the left hand side we have immediately
  $\< \varphi(\bar \ra) \> = \< \psi(\bar \ra) \>$.
  For the right hand side, we have
  $|\[\varphi\]-\[\psi\]| = \[|\varphi-\psi|\]$, whereby $T_0^{Ra} \vdash \[\varphi\] = \[\psi\]$.
  We may therefore assume that $\varphi$ is in prenex form.
  We now proceed by induction on the number of quantifiers.
  If $\varphi$ is quantifier-free then
  $\bigl[ \< \varphi(\bar \ra) \> \bigr] =
  \[ \varphi([\bar \ra]) \]$
  by construction.
  For the induction step, recall that
  \begin{gather*}
    \[ \sinf_y \varphi([\bar \ra],y) \]
    = \sinf_\ry \[ \varphi([\bar \ra],\ry) \]
    = \inf_{\rb\in\widehat\rM} \[ \varphi([\bar \ra],\rb) \]
    = \inf_{\rb\in\rM} \[ \varphi([\bar \ra],[\rb]) \].
  \end{gather*}
  We conclude using the claim and the induction hypothesis.
\end{proof}

Let us go back to the ultra-product example (\fref{exm:UltraProduct}),
where $\rM = \prod \sM$ and
$\widehat \rcM = \prod_\cU \sM$.
By construction
$E\[ \sinf_y \varphi([\bar \ra],y) \]
= \inf_{\rb\in\rM} E\[ \varphi([\bar \ra],[\rb]) \]$.
One also always has
$E\[ \neg\varphi([\bar \ra]) \]
= \neg E\[ \varphi([\bar \ra]) \]$,
$E\[ \half\varphi([\bar \ra]) \]
= \half E\[ \varphi([\bar \ra]) \]$.
Since $E = E_\cU$ is given by an ultra-filter,
we have moreover
$E\[ \varphi([\bar \ra]) \dotminus \psi([\bar \ra]) \]
= E\[ \varphi([\bar \ra]) \] \dotminus
E\[ \psi([\bar \ra]) \]$.
Thus the truth value of $\varphi([\bar \ra])$ in the ultra-product is
precisely $E\[ \varphi([\bar \ra]) \]$ in the sense of the randomised
structure.
Now the last item of \fref{thm:RandomLos} yields the classical
version of Łoś's Theorem:
\begin{gather*}
  \varphi([\bar \ra])
  = E_\cU \bigl[ \< \varphi(\bar \ra) \> \bigr]
  = \lim_\cU \varphi(\ra(\omega)).
\end{gather*}

Let us pursue a little further this analogy with classical
ultra-products.
Classical ultra-product constructions consist of fixing a family
$\sM_\Omega$ and a filter $\cF$ on $\Omega$ with certain desired
properties, then extending this filter to an ultra-filter and taking
the ultra-product.
A filter on $\Omega$ can be viewed as a partial $0/1$ measure:
some sets have measure zero, some measure one, and for some the
measure is not known.
The $[0,1]$-valued analogue is a partial integration
functional on $[0,1]^\Omega$.

\begin{dfn}
  \label{dfn:PartialIntegrationSpace}
  A \emph{partial integration space}
  is a triplet
  $(\Omega,\sA_0,E_0)$ where $\Omega$ is a set,
  $\sA_0 \subseteq [0,1]^\Omega$ is any subset,
  and $E_0 \colon \sA_0 \to [0,1]$
  is a functional satisfying that for every finite sequence
  $\{(X_i,m_i)\}_{i<\ell} \subseteq \sA_0 \times \bZ$ and $k \in \bZ$:
  \begin{gather*}
    \sum m_iX_i \geq k
    \quad \Longrightarrow \quad
    \sum m_iE_0(X_i) \geq k.
  \end{gather*}
  In this case we say that $E_0$ is a
  \emph{partial integration functional}.
\end{dfn}

Clearly every integration functional is a partial integration
functional.
Conversely,

\begin{fct}
  \label{fct:PartialIntegrationFunctional}
  Let $(\Omega,\sA_0,E_0)$ be a partial integration space.
  Then $E_0$ can be extended to a total integration functional $E$ on
  $\sA = [0,1]^\Omega$, rendering
  $(\Omega,\sA,E)$ a (total) integration space.

  Moreover, if $(\Omega,\sA_0,E_0)$ is an integration space, and
  atomless as such, then $(\Omega,\sA,E)$ is atomless as well.
\end{fct}
\begin{proof}
  See \cite[Section~5]{BenYaacov-Keisler:MetricRandom}.
\end{proof}

\begin{dfn}
  \label{dfn:PartialRandomisation}
  A \emph{partial randomisation} based on a family $\sM_\Omega$ is a triplet $(\rM,\sA_0,E_0)$ satisfying all the properties of an ordinary (total) randomisation, with the exception that we do not require that $\< P(\bar \ra) \> \in \sA_0$.
  We say that a partial randomisation is \emph{atomless} if $(\sA_0,E_0) \vDash ARV$.

  By \fref{fct:PartialIntegrationFunctional} we may extend $E_0$ to an integration functional $E$ on $\sA = [0,1]^\Omega$.
  We say that the (full) randomisation $(\rM,\sA) = (\rM,\sA,E)$ is a \emph{totalisation} of $(\rM,\sA_0,E_0)$ and that the associated structure $(\widehat \rcM,\widehat \sA)$ is a structure associated to $(\rM,\sA_0,E_0)$.
  (It is \emph{an} associated structure rather than \emph{the} associated structure because of the arbitrary choices involved.)
\end{dfn}

\begin{dfn}
  \label{dfn:RandomUltraProduct}
  We recall that a \emph{random family of structures}
  $\sM_{(\Omega,\sF,\mu)}$ consists of a family of structures
  $\sM_\Omega = \{\cM_\omega\}_{\omega\in\Omega}$ indexed by a
  probability space $(\Omega,\sF,\mu)$.
  To every such random family we associate a
  natural partial randomisation
  $(\rM,\sA_0,E_0)$ where
  $\rM = \prod \sM$ and $(\Omega,\sA_0,E_0)$ is the integration space
  of $\sF$-measurable functions on $\Omega$.
  It is atomless if and only if
  $(\Omega,\sF,\mu)$ is an atomless probability space.

  If $(\widehat \rcM,\widehat \sA)$ is a structure associated to $(\rM,\sA_0,E_0)$ then we also say that it is a structure associated to the random family $\sM_{(\Omega,\sF,\mu)}$.
\end{dfn}

\begin{cor}
  \label{cor:RandomFamily}
  Let $\sM_{(\Omega,\sF,\mu)}$ be a random family of structures and let $(\rcM,\sA)$ be an associated structure.
  Then for every $\bar \ra$ in $\prod \sM$ and every formula $\varphi(\bar x)$, if $\< \varphi(\bar \ra) \> \in [0,1]^\Omega$ is $\sF$-measurable then
  \begin{gather*}
    E\[ \varphi([\bar \ra]) \]^\rcM
    =
    \int_\Omega \< \varphi(\bar \ra) \>\, d\mu
  \end{gather*}
\end{cor}
\begin{proof}
  Immediate from \fref{thm:RandomLos} and the construction.
\end{proof}

This can be improved to construct extensions containing elements with
desired properties.

\begin{dfn}
  An embedding
  $\sigma\colon(\rcM,\sA) \to (\rcM_1,\sA_1)$ will be called a
  $\[\cdot\]$-embedding if
  $\sigma\[ \varphi(\bar \ra) \]^\rcM
  = \[ \varphi(\sigma\bar \ra) \]^{\rcM_1}$
  for every
  $\bar \ra \in \rM^n$ and formula $\varphi(\bar x)$.
\end{dfn}

\begin{dfn}
  A \emph{morphism} of integration spaces
  $\pi\colon (\Omega',\sA',E') \to (\Omega,\sA,E)$ is a projection
  $\pi \colon \Omega' \twoheadrightarrow \Omega$
  such that
  $X \circ \pi \in \sA'$ and $E'(X \circ \pi) = E(X)$
  for all $X \in \sA$.
\end{dfn}

\begin{cor}
  \label{cor:RandomFamilyExtension}
  Let $(\rcM,\sA) \vDash T_0^{Ra}$ with canonical representation $(\rM,\sA)$ based on $\sM_{(\Omega,\sF,\mu)}$, so in particular $\sA = C(\Omega,[0,1])$.

  Let $\pi\colon (\Omega',\sA_0',E_0') \to (\Omega,\sA,E)$ be a morphism of integration spaces and let $\sM'_{\Omega'} = \{\cM'_{\omega'}\}_{\omega'\in\Omega'}$ be a family of elementary extensions $\cM_{\pi\omega'} \preceq \cM'_{\omega'}$.
  Set $\rM' = \prod \sM'_{\Omega'}$, $\sA' = [0,1]^{\Omega'}$, and for $\ra \in \rM$ and $X \in \sA$ define
  \begin{gather*}
    \sigma\ra = \ra \circ \pi = (\omega' \mapsto \ra(\pi\omega')) \in \rM',
    \qquad
    \sigma X = X \circ \pi \in \sA'.
  \end{gather*}
  Let $(\widehat{\rcM'},\widehat{\sA'})$ be an associated structure to the partial randomisation $(\rM',\sA_0',E_0')$, and let $[\sigma] \colon (\rcM,\sA) \to (\widehat{\rcM'},\widehat{\sA'})$ be the map $\ra \mapsto [\sigma\ra]$, $X \mapsto [\sigma X]$.
  Then
  \begin{enumerate}
  \item
    The map $[\sigma]$ is a $\[\cdot\]$-embedding.
  \item For every $\bar \ra$ in $\prod \sM'$ and every $\varphi(\bar x)$, if $\< \varphi(\bar \ra) \> \in \sA_0'$ then
    \begin{gather*}
      E\[ \varphi([\bar \ra]) \]^{\widehat{\rcM'}} = E_0' \< \varphi(\bar \ra) \>.
    \end{gather*}
  \end{enumerate}
\end{cor}
\begin{proof}
  For the first item it is easy to check that $[\sigma]$ is indeed an embedding.
  In order to see that $[\sigma]$ is a $\[\cdot\]$-embedding let $\bar \ra \in \rM^n$ and let $\varphi(\bar x)$ be a formula.
  Then $\< \varphi(\bar \ra) \>
  = \[ \varphi(\bar \ra) \] \in \sA
  \subseteq [0,1]^\Omega$ by \fref{thm:Satisfaction}, so
  \begin{gather*}
    \bigl[ \sigma\[ \varphi(\bar \ra) \] \bigr]
    = \bigl[ \sigma\< \varphi(\bar \ra) \> \bigr]
    = \bigl[ \< \varphi(\sigma \bar \ra) \> \bigr]
    = \[ \varphi([\sigma \bar \ra]) \].
  \end{gather*}
  The second item is an immediate consequence of \fref{thm:RandomLos}.
\end{proof}

\subsection{Quantifier elimination and types}

Let $T_0^R$ consist of $T_0^{Ra}$ along with the atomlessness axiom ARV.
In other words, $T_0^R$ consists of the theory $ARV$ for the auxiliary sort plus axioms R1-3.
Similarly, we define $T^R = T^{Ra} + {\rm ARV} = T_0^R \cup \{\[\varphi\] = 0\}_{\varphi\in T}$.

\begin{exm}
  \label{exm:AssociatedModelTR}
  Let $\cM \vDash T$ and let $(\Omega,\sF,\mu)$ be any atomless probability space.
  Let $(\rcM,\sA)$ be an associated structure to the constant random family $\sM_{(\Omega,\sF,\mu)} = \{\cM\}_{\omega \in \Omega}$.
  Then $(\rcM,\sA) \vDash T^{Ra}$ by \fref{cor:RandomFamily} and $\sA$ is atomless, whereby $(\rcM,\sA) \vDash T^R$.
\end{exm}

\begin{lem}
  \label{lem:StrongCompanions}
  Every model $(\rcM,\sA) \vDash T^{Ra}$ admits a $\[\cdot\]$-embedding $\sigma\colon (\rcM,\sA) \to (\rcM^1,\sA^1) \vDash T^R$.
  In particular, $T^{Ra}$ and $T^R$ are companions (which, as in classical logic, means that every model of one embeds in a model of the other, or equivalently, that the two theories have the same universal consequences $\sup_{\bar x} \varphi$ for quantifier-free $\varphi$).
\end{lem}
\begin{proof}
  Let $([0,1],\sB,\lambda)$ denote the Lebesgue measure on $[0,1]$.
  Apply \fref{cor:RandomFamilyExtension} to $(\Omega',\sF',\mu')
  = (\Omega,\sF,\mu) \times ([0,1],\sB,\lambda)$ and $\cM'_{\omega,r} = \cM_\omega$.
  The resulting embedding $\sigma\colon (\rcM,\sA) \to (\rcM^1,\sA^1)$ is a $\[\cdot\]$-embedding and $\sA^1$ is atomless.
  If $\varphi \in T$ is a sentence then $\[ \varphi \]^{\rcM^1}
  = \sigma\[ \varphi \]^\rcM = \sigma 0 = 0$.
  Thus $(\rcM^1,\sA^1) \vDash T^R$, as desired.
\end{proof}

Let us now fix an $\cL$-theory $T$.
As usual, $\tS_n(T)$ (or sometimes $\tS_{\bar x}(T)$) denotes the space of $n$-types of $T$.
Similarly, $\tS_n(T^R)$ (or $\tS_{\bar \rx}(T^R)$) denotes the space of $n$-types of the $\cL^R$-theory $T^R$.

Let us fix some additional notation.
For a compact Hausdorff space $X$, let $\fR(X)$ denote the space of regular Borel probability measures on $X$.
For $\varphi \in C(X,\bC)$ and $\mu \in \fR(X)$ let $\langle\varphi,\mu\rangle = \int \varphi \, d\mu$ and equip $\fR(X)$ with the weak topology, namely $\mu_s \to \mu$ if $\langle\varphi,\mu_s\rangle \to \langle\varphi,\mu\rangle$ for all $\varphi$.
It is a classical (and easy) fact that this renders $\fR(X)$ a compact Hausdorff space as well.

Let $\rp(\bar \rx) \in \tS_n(T^R)$.
It is not difficult to verify (e.g., using the Riesz Representation
Theorem) that there exists a unique regular Borel probability measure
$\nu_\rp \in \fR(\tS_n(T))$ characterised by the identity
$E\[\varphi(\bar \rx)\]^\rp
= \langle\varphi,\nu_\rp\rangle$
for every $\cL$-formula $\varphi(\bar x)$.
The map $\rp \mapsto \nu_\rp$ is continuous by definition of the topology
on $\fR(\tS_n(T))$.

We next claim that $\rp \mapsto \nu_\rp$ is surjective.
Indeed, let $\mu \in \fR(\tS_n(T))$.
For each $p \in \tS_n(T)$
choose a model $\cM_p$ and a realisation
$\bar a_p \in M_p^n$ of $p$
(we do not assume that $T$ is
complete so $\cM_p$ may have to vary with $p$).
Let $(\rcM,\sA)$ be a structure associated to the
random family
$\sM = \sM_{(\tS_n(T),\mu)} = \{\cM_p\}_{p\in\tS_n(T)}$.
Let $\bar \ra \in \prod \sM$ be given by
$\bar \ra(p) = \bar a_p$.
By \fref{cor:RandomFamily}, for every formula
$\varphi(\bar x)$:
\begin{gather*}
  E\[\varphi([\bar \ra])\]
  = E[ \< \varphi(\bar \ra) \> ]
  = \langle\varphi,\mu\rangle.
\end{gather*}
In particular, if $\varphi \in T$ is a sentence then $E\[\varphi\] = 0$, so $(\rcM,\sA) \vDash T^{Ra}$.
By \fref{lem:StrongCompanions} we can embed $(\rcM,\sA)$ in a model $(\rcM^1,\sA^1) \vDash T^R$, and if $\rp = \tp^{\rcM^1}(\bar \ra)$ then $\nu_\rp = \mu$.
We argued above for types in finitely many variables, but in exactly the same manner we associate to each $\rp \in \tS_I(T^R)$ a regular Borel probability measure $\nu_\rp \in \fR(\tS_I(T))$ and this map is surjective, for an arbitrary index set $I$.

For quantifier elimination we shall require the following fact from \cite{Bollobas-Varopoulos:Representation}.

\begin{fct}
  \label{fct:ProbabilityHall}
  Let $S$ be any set, $(\Omega,\sF,\mu)$ an atomless probability space.
  For each $x \in S$ let us be given a \emph{weight} $w_x \geq 0$ and an event $C_x \in \sF$.
  For $T \subseteq S$ let $w_T = \sum_{x\in T} w_x$, $C_T = \bigcup_{x\in T} C_x$.
  Then the following are equivalent:
  \begin{enumerate}
  \item For all $T \subseteq S$: $\mu(C_T) \geq w_T$.
  \item There exists a disjoint family $\{D_x\}_{x\in S}$ such that $D_x \subseteq C_x$ and $\mu(D_x) = w_x$.
  \end{enumerate}
  If $w_S = 1$ then $\{D_x\}_{x\in S}$ is a partition of $\Omega$ (up to null measure).
\end{fct}

\begin{lem}
  \label{lem:MeasureVariableExtension}
  Let $(\rcM,\sA) \vDash T_0^R$ be $\aleph_0$-saturated, $\bar \ra \in \rM^n$,
  and let $\nu_{\bar \ra}$ be an abbreviation for $\nu_{\tp(\bar \ra)}$.
  Let $\theta\colon \tS_{n+1}(\cL) \to \tS_n(\cL)$
  be the restriction to the first
  $n$ variables.
  Then:
  \begin{enumerate}
  \item For every $\rb \in \rM$, $\nu_{\bar \ra}$ is the image measure of
    $\nu_{\bar \ra,\rb}$ under $\theta$.
  \item Conversely, let $\eta \in \fR(\tS_{n+1}(\cL))$
    by such that its image measure under $\theta$ is $\nu_{\bar \ra}$.
    Then there is $\rb \in \rM$ such that $\eta = \nu_{\bar \ra,\rb}$.
  \end{enumerate}
\end{lem}
\begin{proof}
  The first item is immediate.
  For the second, it is enough to show that for every finite family
  $\varphi_i(\bar x,y)$, $i < \ell$, and for every $\varepsilon > 0$, there is
  $\rb \in \rM$ such that
  $\big|
  \langle\varphi_i,\eta\rangle
  - E\[\varphi_i(\bar \ra,\rb)\]
  \big|
  < \varepsilon$ for $i < \ell$.

  Let $S = \{s_j\}_{j<k}$ be a partition of $[0,1]^\ell$ into finitely many
  Borel subsets, $\diam(s_i) < \varepsilon$.
  For $j < k$ let $w_j = \eta\{ \bar \varphi \in s_j \}$.
  Choose also $\bar t_j \in s_j$ and let
  $\psi_j = \bigvee_{i<\ell} |\varphi_i-t_{j,i}|$.
  Notice that
  \begin{gather*}
    \left| \langle\varphi_i,\eta\rangle - \sum_{j<k} w_jt_{j,i} \right|
    \leq \sum_{j<k} w_j\diam(s_j) < \varepsilon.
  \end{gather*}
  Let $C_j \in \sF$ be the event
  $\bigl\{ \[ \inf_y \psi_j(\bar \ra,y) \] < \varepsilon \}$.
  Following the notations of \fref{fct:ProbabilityHall},
  we claim that $\mu(C_T) \geq w_T$ for all $T \subseteq k$.
  Indeed, notice that
  $\{\bar \varphi \in s_j\}
  \subseteq \theta^{-1}\{\psi_j < \varepsilon\}$,
  whereby:
  \begin{align*}
    w_T &
    = \sum_{j\in T} \eta\{\bar \varphi \in s_j\}
    = \eta\left( \bigcup_{j\in T} \{\bar \varphi \in s_j\} \right)
    \leq \eta\left( \bigcup_{j\in T} \theta^{-1}\{\psi_j < \varepsilon\} \right)
    \\ &
    = \nu_{\bar \ra}\left( \bigcup_{j\in T} \{\psi_j < \varepsilon\} \right)
    = \mu(C_T).
  \end{align*}
  By \fref{fct:ProbabilityHall} there are events $D_j \subseteq C_j$ such that
  $w_T = \mu(D_T)$ for all $T \subseteq k$.
  Since the total weight is one, $\{D_j\}_{j<k}$ is a partition.
  By \fref{lem:RandomVarQuantifier} and saturation of $\rcM$ there are
  $\rb_j \in \rM$ such that
  $\[\inf_y \psi_j(\bar \ra,y)\]
  = \[ \psi_j(\bar \ra,\rb_j) \]$.
  Notice that:
  \begin{gather*}
    \[ |\varphi_i(\bar \ra,\rb_j) - t_{j,i}| \] \mathbf{1}_{D_j}
    \leq \[\inf_y \psi_j(\bar \ra,y)\] \mathbf{1}_{C_j}
    < \varepsilon.
  \end{gather*}
  Let
  $\rb = \big\langle D_0,\rb_0,\langle D_1,\rb_1,\ldots\rangle \big\rangle$,
  i.e., $\rb(\omega) = \rb_j(\omega)$ when $\omega \in D_j$.
  Now:
  \begin{align*}
    \left|
      \sum_{j<k} w_jt_{j,i}
      - E\[\varphi_i(\bar \ra,\rb)\]
    \right| &
    \leq \sum_{j<k} \left|
      w_jt_{j,i}
      - E\bigl(
      \[\varphi_i(\bar \ra,\rb)\] \mathbf{1}_{D_j}
      \bigr)
    \right|
    \\ &
    = \sum_{j<k} \left|
      E\Bigl[ \bigl(
      t_{j,i} - \[\varphi_i(\bar \ra,\rb_j)\]
      \bigr) \mathbf{1}_{D_j} \Bigr]
    \right|
    \\ &
    \leq E\left[\sum_{j<k}
      \bigl|
      t_{j,i} - \[\varphi_i(\bar \ra,\rb_j)\]
      \bigr|
      \mathbf{1}_{D_j}
    \right]
    < \varepsilon.
  \end{align*}
  Thus
  $\big|
  \langle\varphi_i,\eta\rangle
  - E\[\varphi_i(\bar \ra,\rb)\]
  \big|
  < 2\varepsilon$,
  which is good enough.
\end{proof}

\begin{thm}
  \label{thm:TRTypes}
  \begin{enumerate}
  \item The theories of the form $T^R$ (and in particular
    $T_0^R$) eliminate quantifiers in the main sort down to formulae
    of the form $E\[\varphi(\bar \rx)\]$.
  \item The map $\rp \mapsto \nu_\rp$ defined by
    $\langle\varphi,\nu_\rp\rangle = E\[\varphi\]^\rp$
    induces a homeomorphism
    $\tS_{\bar \rx}(T^R) \simeq \fR(\tS_{\bar x}(T))$.
  \item Let $f\colon n \to m$ be any map.
    Let $f^*\colon \tS_m(T) \to \tS_n(T)$ be the map
    $\tp(a_0,\ldots,a_{m-1}) \mapsto \tp(a_{f(0)},\ldots,a_{f(n-1)})$
    and similarly
    $f^{*,R}\colon \tS_m(T^R) \to \tS_n(T^R)$.
    Let $\tilde f^*\colon \fR(\tS_m(T)) \to \fR(\tS_n(T))$ be the image
    measure map corresponding to $\tilde f^*$.
    Then the following diagram commutes:
    \begin{gather*}
      \begindc{\commdiag}[60]
      \obj(1,2)[SmR]{$\tS_m(T^R)$}
      \obj(1,1)[SnR]{$\tS_n(T^R)$}
      \obj(3,2)[RSm]{$\fR(\tS_m(T))$}
      \obj(3,1)[RSn]{$\fR(\tS_n(T))$}
      \mor{SmR}{SnR}{$f^{*,R}$}
      \mor{RSm}{RSn}{$\tilde f^*$}
      \mor{SmR}{RSm}{$\cong$}
      \mor{SnR}{RSn}{$\cong$}
      \enddc
    \end{gather*}
  \item The completions of $T^R$ are in bijection with
    regular Borel probability measures on the space of completions of
    $T$.
    In particular, if $T$ is complete then so is $T^R$.
  \end{enumerate}
\end{thm}
\begin{proof}
  The first item follows from
  \fref{lem:MeasureVariableExtension} via a standard back-and-forth
  argument.
  For the second item, we have already seen that the map
  $\rp \mapsto \nu_\rp$ is continuous and surjective.
  From the first item it follows that
  it is injective.
  Since both spaces are compact and Hausdorff,
  it is a homeomorphism.
  The third item is easily verified.
  The last item is a special case of the second item for $0$-types.
\end{proof}

\begin{cor}
  Assume that $T$ eliminates quantifiers.
  Then so does $T^R$, and it is the model completion of $T^{Ra}$.
  If $T$ is merely model complete then $T^R$ is model complete as well, and is the model companion of $T^{Ra}$.
\end{cor}
\begin{proof}
  The case where $T$ eliminates quantifiers is easy.
  The case where $T$ is model complete requires a bit more attention to details which we leave to the reader (see \cite[Appendix~A]{BenYaacov:NakanoSpaces} for basic facts regarding model completeness in continuous logic).
\end{proof}

We have described formulae and types on the main sort.
In order to handle the auxiliary sort, add to
$\cL$ a sort
$S_{[0,1]}$ for the set $[0,1]$, equipped with the tautological
predicate $\id\colon S_{[0,1]} \to [0,1]$ and with the usual distance
$d(r,s) = |\id(r)-\id(s)|$.
This is a compact structure and therefore the unique model of its
theory, and adding it to models of $T$ as a new sort does not add any
structure on the original sorts.
Call the resulting signature $\cL_+$ and the corresponding
theory $T_+$.
It is easy to check that $T_+$ eliminates quantifiers if $T$ does.
Passing to $T_+^R$, $\[ \id \]$ is an isometric
bijection between the main sort $S_{[0,1]}^R$ of $T_+^R$ and the
auxiliary sort, so questions about types and definability in the
auxiliary sort can be settled by applying \fref{thm:TRTypes}
to the sort $S_{[0,1]}^R$.

\begin{cor}
  \label{cor:AuxiliarySortDefinability}
  Every $\cL^R$-definable predicate on
  the auxiliary sort of $T^R$, possibly with parameters
  $\bar \ra$ from the main sorts, is equivalent
  to one in the pure language of the
  auxiliary sort and with parameters in
  \begin{gather*}
    \sigma(\bar \ra) =
    \sigma\bigl(
    \[\varphi(\bar \ra)\]
    \bigr)_{\varphi(\bar x)\in \cL_{\omega,\omega}}
    \subseteq \sF.
  \end{gather*}
  Consequently, the auxiliary sort is stable and stably embedded in
  models of $T^R$, and if $\bar X$ is a tuple in the auxiliary sort then
  \begin{gather*}
    \tp(\bar X/\bar \ra) \quad \equiv \quad
    \tp(\bar X/\sigma(\bar \ra)).
  \end{gather*}
\end{cor}
\begin{proof}
  We may assume that $T$ eliminates
  quantifiers, in which case so does $T_+$ and therefore $T_+^R$.
  It is therefore enough to show that for a tuple of variables
  $\bar \rr$ in the sort $S_{[0,1]}^R$ and for any possible additional
  parameters $\bar \ra$, any atomic formula
  in $\bar \rr$ and $\bar \ra$ is equivalent to a formula entirely in
  $S_{[0,1]}^R$, possibly using parameters in
  $\[ \id \]^{-1}(\sigma(\bar \ra))$.
  Given the minimalistic structure we put on $S_{[0,1]}$,
  such an atomic formula can either involve precisely one
  free variable $\rr_i$ or some of the parameters but no free
  variable.
  In the first case we have $\[ \id(\rr_i)\]$ which is
  as desired.
  In the second we have
  $\[ \varphi(\bar \ra) \]$ where
  $\varphi(\bar x)$ is an atomic $\cL$-formula.
  In this case let
  $X = \[ \id \]^{-1}
  \circ \[ \varphi(\bar \ra) \] \in S_{[0,1]}^R$,
  so
  $\[ \varphi(\bar \ra) \] =
  \[ \id(X) \]$, and the latter is again as
  desired.
\end{proof}

\subsection{Types in $T^R$ when $T$ is incomplete}
\label{sec:TypesIncomleteT}

\fref{thm:TRTypes} provides us with a complete description of types in
$T^R$, whether $T$ is complete or not.
In various situations we shall encounter later on, this description
turns out to be much more useful when $T$ is complete.
What follows here is a brief discussion of the general case and a
reduction of sorts to the special case of a complete theory.

Let $T$ be an incomplete theory and let $\rp \in \tS_n(T^R)$.
By \fref{thm:TRTypes} we may identify $\rp$ with a regular Borel
probability measure $\nu_\rp \in \fR(\tS_n(T))$.
Let $\sA_\rp = L^1\bigl( (\tS_n(T), \nu_\rp),[0,1] \bigr)$ and let
$(\Omega_\rp,\mu_\rp)$ be its Stone space.
We have a natural identification of
$C(\Omega_\rp,[0,1])$ with $\sA_\rp$, so in particular every
$n$-ary $\cL$-formula $\varphi(\bar x)$ gives rise to a continuous function
$\varphi\colon \tS_n(T) \to [0,1]$ with image
$\overline \varphi_\rp \in \sA_\rp = C(\Omega_\rp,[0,1])$.
Thus for every $\omega \in \Omega_\rp$ we may define a complete type
$\pi_\rp\omega \in \tS_n(T)$ by
$\varphi^{\pi_\rp\omega} = \overline\varphi_\rp(\omega)$.
We obtain a map
$\pi_\rp\colon (\Omega_\rp,\mu_\rp)\to(\tS_n(T),\nu_\rp)$ which is
continuous and sends $\mu_\rp$ to $\nu_\rp$ (as an image measure).
It follows that the image of
$\pi_\rp$ is precisely the support of
$\nu_\rp$ there, i.e., the smallest closed measure one set.
This discussion holds in particular when $n = 0$, i.e.,
when $\rT = \rp \in \tS_0(T^R)$ is a completion of $T^R$.

Let now $\rT$ be a completion of $T^R$ and
$\rp(\bar \rx) \in \tS_n(\rT)$.
There is a natural $\cL_{RV}$-inclusion $\sA_\rT \subseteq \sA_\rp$
giving rise to a projection
$(\Omega_\rp,\mu_\rp) \to (\Omega_\rT,\mu_\rT)$
where $\mu_\rT$ is the image of $\mu_\rp$.
As in \cite[Section~5]{BenYaacov-Keisler:MetricRandom} this projection
gives rise to a conditional expectation map
$\bE[\cdot|\rT]\colon \sA_\rp \to \sA_\rT$.
In particular, to every formula $\varphi(\bar x)$ we associated
$\overline \varphi_\rp$ which in turn gets sent to
$\bE[\overline \varphi_\rp|\rT] \in \sA_\rT = C(\Omega_\rT,[0,1])$.
Let us fix $\omega \in \Omega_\rT$.
It is not difficult to check that
$\varphi \mapsto \bE[\overline \varphi_\rp|\rT](\omega)$
is an integration functional on $\tS_n(T)$.
Therefore there exists a unique type $\rp_\omega \in \tS_n(T^R)$
verifying for all $\varphi(\bar x)$:
\begin{gather*}
  E\[ \varphi \]^{\rp_\omega}
  = \langle \varphi,\nu_{\rp_\omega} \rangle
  =
  \bE[\overline \varphi_\rp|\rT](\omega).
\end{gather*}
The map $\omega \mapsto \rp_\omega$ has the following properties:
\begin{enumerate}
\item It is determined by $\rp$ (in particular, the completion $\rT$
  is determined by $\rp$).
\item Conversely, it determines $\rp$ as follows:
  \begin{gather*}
    E\[ \varphi \]^\rp
    =
    \int_{\Omega_\rT} E\[ \varphi \]^{\rp_\omega}
    \,d\mu_\rT(\omega).
  \end{gather*}
\item Let $\pi_\rT\colon \Omega_\rT \to \tS_0(T)$ be as in the
  previous paragraph, associating to each $\omega \in \Omega_\rT$ a
  completion $\pi_\rT\omega$ of $T$.
  Then
  $\rp_\omega \in \tS_n\bigl( (\pi_\rT\omega)^R \bigr)$.
\item The map $\omega \mapsto \rp_\omega$ is continuous in the
  appropriate weak topology.
  Specifically, for every formula $\varphi(\bar x)$ the map
  $\omega \mapsto \langle \varphi,\nu_{\rp_\omega} \rangle$
  is continuous on $\Omega_\rT$.
\end{enumerate}
We therefore write
\begin{gather*}
  \rp = \int_{\Omega_\rT} \rp_\omega\,d\mu_\rT(\omega),
\end{gather*}
saying this is an integral of a continuous family.
(Conversely, every integral of a continuous or even measurable
family of $T^R$-types gives rise to a $T^R$-type.)

A special case of this situation is a type over parameters.
Let $(\rcM,\sA) \vDash T^R$, $\rA \subseteq \rM$ and
$\rp \in \tS_n(\rA)$.
Let us enumerate $\rA = \{\ra_\alpha\}_{\alpha\in I}$.
Let $A = \{a_\alpha\}_{\alpha\in I}$ be a set of new constant symbols
and let $\cL_A = \cL \cup A$.
Even if $T$ is a complete $\cL$-theory it is incomplete as an
$\cL_A$-theory.
We view $(\rcM,\sA)$ as an $\cL_A^R$-structure naming $\rA$ by $A$.
It is then the model of a complete $\cL_A^R$-theory $\rT$ and
$\rp \in \tS_n(\rT)$.
Each $\omega \in \Omega_\rT$ gives rise to an $\cL_A$-completion
$\pi_\rT\omega$ of $T$.
In other words, each $\omega$ determines, so to speak, the $\cL$-type
of the constants $A$.
Let $A_\omega$ be an actual set in a model of $T$ realising this type.
Then
$\rp_\omega \in \fR\bigl(\tS_n\bigl( \pi_\rT\omega \bigr)\bigr)
= \fR(\tS_n(A_\omega))$.

\section{Preservation and non-preservation results}
\label{sec:Preservation}

\subsection{Categoricity}

The theory $ARV$ is $\aleph_0$-categorical but not uncountably
categorical, so this is the most we can hope for from $T^R$.
We shall use the following criterion for $\aleph_0$-categoricity.

\begin{fct}[Ryll-Nardzewski Theorem for metric structures]
  A complete countable theory $T$ is $\aleph_0$-categorical if and
  only if $\tS_n(T)$ is metrically compact for all $n$,
  if and only if the logic topology on $\tS_n(T)$ coincides with the
  metric for all $n$.
\end{fct}
\begin{proof}
  This was originally stated and proved by C.\ Ward Henson for Banach
  space structures.
  For the proof in continuous logic see
  \cite[Fact~1.14]{BenYaacov-Usvyatsov:dFiniteness}.
  Notice that it is customary to exclude the case of a complete theory
  with a compact model (which is its unique model)
  from the definition of $\aleph_0$-categoricity, as well as from the
  statement of this theorem, but the theorem holds as stated if we do
  not.
\end{proof}

\begin{thm}
  Assume $T$ is a complete $\aleph_0$-categorical theory in a
  countable language.
  Then so it $T^R$.
\end{thm}
\begin{proof}
  It is enough to show that $\tS_n(T^R)$ is totally bounded, i.e.,
  that it can be covered by finitely many $\varepsilon$-balls
  for every $\varepsilon > 0$.
  Let us therefore fix $\varepsilon > 0$.
  By assumption we can cover $\tS_n(T)$ with finitely many
  $\varepsilon$-balls, say
  $\tS_n(T) = \bigcup_{i<k} B(p_i,\varepsilon)$.
  Fix $N > \frac{k}{\varepsilon}$, and let
  $R = \{ \bar m \in \bN^k\colon \sum m_i = N\}$.
  Then $R$ is finite, and for every $\bar m \in R$ we may define
  $\rp_{\bar m} = \sum \frac{m_i}{N}p_i \in \tS_n(T^R)$.
  Let also
  $C_i =
  B(p_i,\varepsilon) \setminus \bigcup_{j<i} B(p_j,\varepsilon)$,
  so $\tS_n(T) = \bigcup_i C_i$ is a partition of $\tS_n(T)$ into a
  finite disjoint union of Borel sets of diameter $\leq \varepsilon$.

  Now let $\rq \in \tS_n(T^R)$ be any type.
  Find a tuple $\bar m \in R$ such that
  $E = \| \bar m/N - (\nu_\rq(C_i))_{i<k} \|_1$ is minimal.
  We can do this so that at each coördinate the difference is at most
  $\frac{1}{N}$, so $E < \frac{k}{N} < \varepsilon$.
  We claim that $d(\rq,\rp_{\bar m}) < 2\varepsilon$, which will
  conclude the proof.

  Let $\ra \in \rcM$ realise $\rq$, and as usual let us identify
  $\rcM$ with its canonical representation, based on
  $\sM_{(\Omega,\mu)}$.
  Let
  $D_i = \{\omega \in \Omega\colon \tp(\ra(\omega)) \in C_i\}$,
  so $\mu(D_i) = \nu_\rq(C_i)$,
  and $\Omega = \bigcup D_i$ is a partition of $\Omega$ into disjoint
  Borel sets.
  We can now choose another such partition
  $\Omega = \bigcup D_i'$ such that
  each $D_i'$ is comparable with $D_i$
  (i.e., either $D_i \subseteq D_i'$ or $D_i' \subseteq D_i$)
  and $\mu(D_i') = \frac{m_i}{N}$,
  so
  $\mu( D_i \triangle D_i' ) = | \frac{m_i}{N} - \nu_\rq(C_i) |$.
  For each $\omega \in D_i'$ choose
  $\cM'_\omega \succeq \cM_\omega$ and
  $\rb(\omega) \in \cM'_\omega$ realising $p_i$.
  If $\omega \in D_i \cap D_i'$ then we arrange that in addition
  $d\bigl( \rb(\omega), \ra(\omega) ) < \varepsilon$.
  Apply \fref{cor:RandomFamilyExtension} to obtain an elementary
  extension $\rcM' \succeq \rcM$ and
  $\rb \in \rcM'$
  such that $\tp(\rb) = \rp_{\bar m}$ and
  $d(\rb,\ra) \leq \varepsilon(1-E) + E < 2\varepsilon$,
  as desired.
\end{proof}

\begin{cor}
  Assume $T$ is a countable theory, possibly incomplete, with
  countably many completions, all of which are
  $\aleph_0$-categorical.
  Then every completion of $T^R$ is $\aleph_0$-categorical.
\end{cor}
\begin{proof}
  Let $\{T_n\}_{n\in\alpha}$, denote the set of completions of $T$, where $\alpha \leq \aleph_0$.
  The completions of $T$ are in bijection with measures on $\alpha$, namely with sequences $\bar \lambda \in [0,1]^\alpha$ such that $\sum \lambda_n = 1$.
  Every model $\rcM$ of such a completion can be identified with a combination $\sum \lambda_n \rcM_n$ where $\rcM_n \vDash (T_n)^R$, and is uniquely determined by $\rcM$ except where $\lambda_n = 0$.
  If $\rcM$ is separable then so is $\rcM_n$ (when $\lambda_n > 0$), whence the uniqueness of $\rcM$.
\end{proof}

Of course, in this case $T^R$ may admit continuum many completions,
and yet it is not too difficult to see that every completion of
$(T^R)^R$ is still $\aleph_0$-categorical.
On the other hand, the are theories $T$ with uncountably many
completions, all of which are $\aleph_0$-categorical, such that
$T^R$ admits a non $\aleph_0$-categorical completion.
Indeed, let $T$ be the classical theory saying that there exist
precisely $2$ elements, in a language with constants
$a$ and $b_n$ for $n \in \bN$.
Using $a$ as reference, a completion of $T$ is determined by whether
$b_n = a$ or not for each $n$, so the space of completions of $T$ is
homeomorphic to $2^\bN$.
Let $\rT$ be the completion of $T^R$ saying that
$\[ b_n = a \]$ are independent events all of measure $\half$.
Let $\rp_n \in \tS_1(\rT)$ be the type $\rx = b_n$.
Then $d(\rp_n,\rp_m) = \half$ for all $n \neq m$ and $\tS_1(\rT)$ is
not metrically compact.

\subsection{Stability}
\label{sec:PreservationStability}

For all facts regarding stability in continuous logic, and in particular local stability, we refer the reader to \cite{BenYaacov-Usvyatsov:CFO}.
For topometric Cantor-Bendixson ranks see \cite{BenYaacov:TopometricSpacesAndPerturbations}.

When proving the preservation of stability in \cite{BenYaacov-Keisler:MetricRandom} we considered $\varphi$-types over arbitrary sets in models of $T$ and of $T^R$, calculating
averages over the finite set of non forking extensions of such types.
In doing so we proved not only that the randomisation of a stable theory is stable, but also that in such a randomised theory all types over sets (in sorts of the original theory) were stationary.

In continuous logic the situation is, at least on the surface, much
more complicated.
Assume $A \subseteq M$, $p \in \tS_\varphi(A)$, and let
$P \subseteq \tS_\varphi(M)$ be the set of non forking extensions of
$p$.
Rather than being a finite set, as in classical logic, $P$ is merely
known to be a transitive compact metric space (in the standard
metric on $\tS_\varphi(M)$, namely
$d(q,q') =
\sup \bigl\{|\varphi(x,b)^q-\varphi(x,b)^{q'}|\colon
b \in M \bigr\}$).
By \emph{transitive} we mean that the action of the isometry group of
$P$ is transitive, which leads to the existence of a canonical
probability measure on $P$ and thus to a canonical notion of an
average value of a function on $P$.
With this notion of average we could, in principle,
translate the entire argument of
\cite{BenYaacov-Keisler:MetricRandom} to the case where $T$ is
continuous.
However, calculating averages over a transitive compact metric space
is significantly more involved than merely averaging over a finite
set, rendering the translated argument quite difficult to follow.

We therefore choose to split the argument in two, and at a first time
restrict our attention to types over models, in which case the non
forking extension is unique and no averaging is required.
In \fref{sec:LascarTypes} we prove quite independently that for any
theory $T$ (stable or not), types in $T^R$ coincide with Lascar types.
It follows that if $T$ is stable then all types in $T^R$ are
stationary.

As in \cite{BenYaacov-Keisler:MetricRandom} we shall use Shelah ranks,
this time adapted to continuous logic.
Let us fix for the time being a monster model
$\fM$ containing all the parameters under consideration.
We define the \emph{$(k,\varphi)$-rank}
of a partial type $\pi(x)$,
denoted $R_k(\pi,\varphi)$, and its
\emph{multiplicity at rank $s$}, denoted
$M_k(\pi,\varphi,s)$:
\begin{enumerate}
\item If $\pi$ is consistent then $R_k(\pi,\varphi) \geq 0$.
\item Having defined when
  $R_k(\pi,\varphi) \geq s$ we define
  $M_k(\pi,\varphi,s)$.
  We say that $M_k(\pi,\varphi,s) \geq M$
  if there are types
  $\pi(x) \subseteq \pi_n(x)$ for $n < M$ such that for every
  $n < m < M$ there exists $b_{nm}$ for which
  \begin{gather*}
    \pi_n(x) \cup \pi_m(x')
    \vdash  |\varphi(x,b_{nm}) - \varphi(x',b_{nm})| \geq 2^{-k},
  \end{gather*}
  and in addition $R_k(\pi_n,\varphi) \geq s$ for all
  $i < M$.
\item If $M_k(\pi,\varphi,s) = \infty$ then
  $R_k(\pi,\varphi) \geq s+1$.
\end{enumerate}
It is not difficult to see that if $[\pi]_\varphi$ denotes the closed set
$\pi$ defines in $\tS_\varphi(\fM)$, then:
\begin{gather*}
  R_k(\pi,\varphi)
  = \CB_{f,2^{-k}}([\pi]_\varphi)
  = \CB_{b,2^{-k}}([\pi]_\varphi),
\end{gather*}
where $\CB_{f,\varepsilon}$ and $\CB_{b,\varepsilon}$ are the
topometric Cantor-Bendixson ranks defined in
\cite[Section~3]{BenYaacov:TopometricSpacesAndPerturbations}.
(Or almost: these are the ranks we would obtain if we replaced
there
``$\leq \varepsilon$'' with ``$<\varepsilon$'' and
``$> \varepsilon$'' with ``$\geq \varepsilon$''.
Since we consider ranks for all $\varepsilon > 0$ this makes no
difference.)

Let $W$ denote a possibly infinite tuple of variables,
$\pi(x,W)$ a partial type and $k,s \in \bN$.
Then $R_k(\pi(x,W),\varphi) \geq s$ is a property of
$W$, holding for $A$ (of the appropriate size)
if $R_k(\pi(x,A),\varphi) \geq s$.
We may think of $R_k(\cdot,\varphi(x,y))$ as a quantifier
binding the variable $x$.
Let also
$R_k(x/W,\varphi) \geq s$ be the property of $xW$ which holds for
$aA$ if $R_k(\tp(a/A),\varphi) \geq s$.

\begin{fct}
  \label{fct:TypeDefinableRank}
  The properties $R_k(\pi(x,W),\varphi) \geq s$
  and $R_k(x/W,\varphi) \geq s$ are type-definable
  (in $W$ and in $xW$, respectively).
\end{fct}
\begin{proof}
  Both are shown using a standard
  ``there exists a tree such that\ldots'' argument.
  The second can be deduced from the first since it may be re-written
  as $R_k( x' \equiv_W x,\varphi(x',y)) \geq s$ where
  $x'$ is the bound variable and $xW$ the parameter variables.
\end{proof}

For $\alpha \leq \omega$ let
$R_{<\alpha}(\pi,\alpha)$ denote the (finite or infinite) sequence
$\bigl( R_k(\pi,\varphi) \bigr)_{k < \alpha}$.
Given a sequence $\sigma \in \bN^\alpha$ and a partial type
$\pi(x)$ let
\begin{gather*}
  \tS_\varphi(\fM)^{(\sigma)}
  \begin{aligned}[t]
    & = \bigl\{ q \in \tS_\varphi(\fM) \colon
    R_{<\alpha}(q, \varphi) \geq \sigma \bigr\}
    \\
    & = \bigl\{ q \in \tS_\varphi(\fM) \colon
    R_k( q, \varphi) \geq \sigma(k) \text{ for all } k < \alpha \bigr\},
  \end{aligned}
  \\
  [\pi]_\varphi^{(\sigma)}
  = [\pi]_\varphi \cap \tS_\varphi(\fM)^{(\sigma)}
  =
  \bigl\{
  q \in \tS_\varphi(\fM) \colon
  R_{<\alpha}( q \cup \pi, \varphi) \geq \sigma
  \bigr\}.
\end{gather*}
We observe that $\tS_\varphi(\fM)^{(\sigma)}$
and therefore $[\pi]_\varphi^{(\sigma)}$ are closed sets
(either directly or using properties of the topometric
Cantor-Bendixson ranks).

Before going further let us recall (e.g., from Bourbaki \cite[Chapter~IV.6]{Bourbaki:Topology1}) that for a topological space $X$, a map $f\colon X \to \bR$ is \emph{upper (respectively, lower) semi-continuous} if $\inf f(A) = \inf f(\overline A)$ (respectively, $\sup f(A) = \sup f(\overline A)$) for every non empty $A \subseteq X$.
We shall require the following easy classical result (see for example Katětov \cite{Katetov:OnRealValuedFunctions}).

\begin{fct}
  \label{fct:SemiContinuity}
  Let $X$ be a compact Hausdorff space.
  \begin{enumerate}
  \item
    A function $f\colon X \to \bR$ is lower semi-continuous if and only if it can be written as the point-wise supremum of a family of continuous functions on $X$.
  \item
    Assume that $f\colon X \to \bR$ is lower semi-continuous, $g\colon X \to \bR$ is upper semi-continuous, and $g < f$.
    Then there is a continuous function $h \colon X \to \bR$ such that $g < h < f$.
  \end{enumerate}
\end{fct}

At this point let us fix $\varepsilon > 0$.
For a finite sequence $\sigma \in \bN^{<\omega}$ we define $\Xi_{\sigma}$ to be the set of all formulae $\xi(x,\bar w)$ such that for any $\bar a \in M$ the diameter of $[\xi(x,\bar a) < 1]^{(\sigma)} \subseteq \tS_\varphi(M)$ is smaller than $\varepsilon$ (analogous to $\Xi_{s,2}$ as defined in
\cite{BenYaacov-Keisler:MetricRandom}).

\begin{lem}
  \label{lem:LocalUniformApproximateDefinition}
  Let $\sigma \in \bN^k$,
  $\xi(x,\bar w) \in \Xi_\sigma$.
  Then there exists a formula
  $\hat \xi_\sigma(y,\bar w)$ such that:
  \begin{gather*}
    \{ \xi(x,\bar w) \leq \half \}
    \cup
    \{ R_{<k}(x/\bar wy,\varphi) \geq \sigma \}
    \vdash |\hat \xi_\sigma(y,\bar w)-\varphi(x,y)| < \varepsilon.
  \end{gather*}
\end{lem}
\begin{proof}
  Let $Y \subseteq \tS_{x,y,\bar w}(T)$ consist of all types
  $q(x,y,\bar w)$ for which the left hand side holds.
  Let $X \subseteq \tS_{x,\bar w}(T)$ consist of all types verifying
  $\xi(x,\bar w) \leq \half$ and
  $R_{<k}(x/\bar w,\varphi) \geq \sigma$.
  The restriction map $\pi\colon Y \to X$ is surjective.

  For $p(x,\bar z) \in X$ let
  $a,\bar c \vDash p$
  and define
  \begin{gather*}
    \overline f(p)
    = \max \bigl\{ \varphi(x,y)^q\colon q \in \pi^{-1}(p) \bigr\}, \\
    \underline f(p)
    = \min \bigl\{ \varphi(x,y)^q\colon q \in \pi^{-1}(p) \bigr\}.
  \end{gather*}
  Let us make a few remarks regarding this definition.
  Since $\pi$ is surjective the set $\pi^{-1}(p)$ is non empty
  and compact.
  The maximum and minimum are therefore attained and
  $\underline f(p) \leq \overline f(q)$.
  Moreover, there are types (not necessarily uniquely determined)
  $\overline p(x), \underline p(x) \in [p(x,\bar c)]^{(\sigma)}$
  such that
  $\varphi(x,b)^{\overline p} = \overline f(p)$
  and
  $\varphi(x,b)^{\underline p} = \underline f(p)$.
  By hypothesis $\xi(x,\bar z)^p \leq \half$,
  so $d(\overline p,\underline p) < \varepsilon$ and thus
  $\overline f(p)  <  \underline f(p) + \varepsilon$.

  Letting $p$ vary over $X$ it is easy to check that
  $\overline f(p)$ is upper semi-continuous
  and similarly $\underline f$ is lower semi-continuous.
  Thus there is a continuous function
  $h\colon X \to [0,1]$ verifying
  $\overline f < h < \underline f + \varepsilon$.
  By Tietze's Extension Theorem there exists a continuous function
  $\tilde h\colon \tS_{x,\bar w}(T) \to [0,1]$
  extending $h$ and we may identify $\tilde h$ with a definable
  predicate $\hat \xi_\sigma(x,\bar w)$ (or, if we insist on having
  an actual formula, we take $\hat \xi_\sigma(x,\bar w)$ to be a
  formula close enough to $h$ so that
  $\overline f < \hat \xi_\sigma(x,\bar w)
  < \underline f + \varepsilon$ on $X$).

  It is left to show that $\hat \xi_\sigma(y,\bar w)$ is as
  desired.
  Indeed, assume that $a,b,\bar c \vDash q \in Y$.
  Then
  $\underline f(q\rest_{x,\bar w})
  \leq \varphi(a,b)
  \leq \overline f(q\rest_{x,\bar w})$,
  whereby
  \begin{gather*}
    \varphi(a,b)
    \leq \overline f(q\rest_{x,\bar w})
    < \hat \xi_\sigma(a,\bar c)
    < \underline f(q\rest_{x,\bar w}) + \varepsilon
    \leq \varphi(a,b) + \varepsilon.
    \qedhere
  \end{gather*}
\end{proof}

We now turn to showing that members of $\Xi_\sigma$ are, in a sense,
plenty enough.

\begin{lem}
  \label{lem:CaptureFewExtensions}
  Let $\cM$ be a model of $T$, $p \in \tS_x(M)$ a type,
  and let $\eta = R_{<\omega}(p,\varphi)$.
  \begin{enumerate}
  \item Let $\fM \succeq \cM$ be a very homogeneous and
    saturated extension and let
    $[p]^{(\eta)} \subseteq \tS_\varphi(\fM)$
    be defined as above.
    Then $[p]^{(\eta)} = \{q\}$ where $q$ is the
    unique non forking $\varphi$-extension of $p$.
  \item There are $k \in \bN$,
    $\xi(x,\bar w) \in \Xi_{\eta\rest_k}$
    and $\bar c \in M$ such that
    $\xi(x,\bar c) \in p$.
    Moreover, for any dense subset $M_0 \subseteq M$ we may arrange
    our choices so that $\bar c \subseteq M_0$.
  \end{enumerate}
\end{lem}
\begin{proof}
  The argument for the first item essentially appears in
  \cite{BenYaacov-Usvyatsov:CFO}, although the Cantor-Bendixson ranks
  used there are different.
  It goes through the following steps.
  The set $[p]^{(\eta)}$ is topologically and
  therefore metrically closed.
  By construction it is non empty and totally bounded,
  and therefore metrically compact.
  Clearly it is also $M$-invariant, and it follows that every
  $q \in [p]^{(\eta)}$
  is definable over $\acl^{eq}(M) = M$.
  We conclude there is a unique such $q$ which follows the definition
  of $p$.

  For the second item consider the following partial type over $M$:
  \begin{multline*}
    p(x)\cup p(x')
    \cup \{|\varphi(x,y)-\varphi(x',y)| \geq \varepsilon\}
    \cup \{R_{<\omega}(x/My,\varphi) \geq \eta\}
    \cup \{R_{<\omega}(x'/My,\varphi) \geq \eta\}.
  \end{multline*}
  By the first item this type is contradictory.
  Let us re-write $p(x)$ as $p(x,M)$ where
  $p(x,W) \in \tS_{x,W}(T)$ is a complete type.
  Then the following is inconsistent:
  \begin{multline*}
    p(x,W)\cup p(x',W)
    \cup \{|\varphi(x,y)-\varphi(x',y)| \geq \varepsilon\}
    \cup \{R_{<\omega}(x/Wy,\varphi) \geq \eta\}
    \cup \{R_{<\omega}(x'/Wy,\varphi) \geq \eta\}.
  \end{multline*}
  Thus by compactness there are $k \in \bN$ and
  $\xi_0(x,\bar w) \in p(x,W)$ such that the following is inconsistent:
  \begin{multline*}
    \{\xi(x,\bar w) \leq 2^{-k}\}
    \cup \{\xi(x',\bar w) \leq 2^{-k}\}
    \cup \{|\varphi(x,y)-\varphi(x',y)| \geq \varepsilon\} \\
    \cup \{R_{<k}(x/\bar w y,\varphi) \geq \eta\rest_k\}
    \cup \{R_{<k}(x'/\bar w y,\varphi) \geq \eta\rest_k\}.
  \end{multline*}
  Let $\xi = \xi_0 \dotplus \cdots \dotplus \xi_0$ ($2^k$ many times)
  and let $\bar c \subseteq M$ correspond to $\bar w \subseteq W$.
  Then $\xi$, $k$ and $\bar c$ are as desired.

  For the moreover part first notice that $p$ is equivalent to its
  restriction to $M_0$ (where $M_0 \subseteq M$ is any dense subset),
  so the argument above holds just as well with $M_0$ in place of
  $M$.
\end{proof}

\begin{lem}
  \label{lem:MeasureOneElementarySubstructure}
  Let $T$ be any theory in a countable language.
  Let $\rcM \vDash T^R$ be a model based on $(\Omega,\mu)$
  and let $\rcM_0 \preceq \rcM$ be a countable elementary
  pre-sub-structure.
  For $\omega \in \Omega$ let
  $\rM_0(\omega) = \{\ra(\omega)\}_{\ra \in \rM_0} \subseteq
  \rM(\omega)$.
  Then
  $\rcM_0(\omega) \preceq \rcM(\omega)$
  as $\cL$-pre-structures
  for all $\omega$ outside a null measure set.
\end{lem}
\begin{proof}
  Let us fix a formula $\varphi(x,\bar w)$ and $\bar \rb \in \rM_0$.
  By \fref{thm:Satisfaction} we have
  $\ssup_x \varphi(x,\bar \rb(\omega))^{\rcM(\omega)}
  =
  \[ \ssup_x \varphi(x,\bar \rb) \]^\rcM(\omega)$
  for all $\omega \in \Omega$,
  where
  $\[ \ssup_x \varphi(x,\bar \rb) \]^\rcM$ is viewed as a continuous
  function $\Omega \to [0,1]$.
  On the other hand we have
  \begin{multline*}
    \[ \ssup_x \varphi(x,\bar \rb) \]^\rcM
    =
    \bigl(
    \ssup_\rx\[ \varphi(\rx,\bar \rb) \]
    \bigr)^{\widehat{\rcM_0}} \\
    =
    \ssup^{L_1} \left\{
      \[ \varphi(\ra,\bar \rb) \]^{\widehat{\rcM_0}}
      \colon \ra \in \widehat{\rM_0}
    \right\}
    =
    \ssup^{L_1} \left\{
      \[ \varphi(\ra,\bar \rb) \]^\rcM
      \colon \ra \in \rM_0
    \right\}.
  \end{multline*}
  Thus we have outside a null measure set
  \begin{gather*}
    \[ \ssup_x \varphi(x,\bar \rb) \]^\rcM(\omega)
    = \sup \left\{
      \varphi(\ra(\omega),\bar \rb(\omega))^{\rcM(\omega)}
      \colon \ra \in \rM_0
    \right\}.
  \end{gather*}
  There are countably many formulae $\varphi(x,\bar \rb)$ to be
  considered, so outside a null measure set the Tarski-Vaught Criterion
  holds and $\rcM_0(\omega) \preceq \rcM(\omega)$.
\end{proof}

\begin{thm}
  Let $\varphi(x,y)$ be a stable formula for a theory $T$.
  Then the formula $E\[\varphi(\rx,\ry)\]$ is stable for $T^R$.
  If $T$ is stable the so is $T^R$.
\end{thm}
\begin{proof}
  We may assume that the language of $T$ is countable (for if not, we may restrict to a sub-language containing just the symbols appearing in $\varphi$).
  It will therefore be enough to show that for every separable model $\rcM \vDash T^R$, every type $\rp \in \tS_\rx(\rM)$ is $E\[\varphi\]$-definable.
  For this purpose let $\ra \in \rcM' \succeq \rcM$ realise $\rp$.
  Let also $\rM_0 = \{\rc_n\}_{n\in\bN} \subseteq \rM$ be a dense pre-sub-structure.
  By \fref{lem:MeasureOneElementarySubstructure} there is a measure one set $\Omega_0 \subseteq \Omega$ such that $\rcM_0(\omega) \preceq \rcM(\omega)$ for all $\omega \in \Omega_0$.

  Let $\Upsilon$ consist of all triplets $(\sigma,\xi,\bar \rc)$ where $\sigma\in \bN^{<\omega}$, $\xi(x,\bar w) \in \Xi_\sigma$ and $\bar \rc \in \rM_0$ with $|\bar w| = |\bar \rc|$.
  For $(\sigma,\xi,\bar \rc) \in \Upsilon$ define subsets of $\Omega$ as follows:
  \begin{gather*}
    A_\sigma = \{\omega\colon
    R_{<|\sigma|}(\ra(\omega)/M_\omega,\varphi) \geq \sigma\},
    \\
    B_{\sigma,\xi,\bar \rc} = \{\omega \in A_\sigma\colon
    \xi(\ra(\omega),\bar \rc(\omega)) < 1\}.
  \end{gather*}
  The set $A_\sigma$ is closed and each $B_{\sigma,\xi,\bar \rc}$ is relatively open in $A_\sigma$, so in particular Borel.
  Moreover, by \fref{lem:CaptureFewExtensions} every $\omega \in \Omega_0$ belongs to some $B_{\sigma,\xi,\bar \rc}$.

  Given all our countability assumptions we may enumerate $\Upsilon = 
    \{(\sigma^m,\xi^m,\bar \rc^m)\}_{m\in\bN}$.
  Let us also write $\xi^m$ explicitly as $\xi^m(x,\bar w^m)$.
  By \fref{lem:LocalUniformApproximateDefinition} there is a formula $\hat \xi^m_{\sigma^m}(y,\bar w^m)$ such that
  \begin{gather*}
    \{ \xi^m(x,\bar w^m) < 1 \}
    \cup
    \{ R_{<|\sigma^m|}(x/\bar w^my,\varphi) \geq \sigma^m \}
    \vdash |\hat \xi^m_{\sigma^m}(y,\bar w^m)-\varphi(x,y)| < \varepsilon.
  \end{gather*}
  For $m \in \bN$ let $D_m = B_{\sigma^m,\xi^m,\bar \rc^m}
  \setminus \bigcup_{k<m} B_{\sigma^k,\xi^k,\bar \rc^k}$.
  Then $\{D_m\}_{m\in \bN}$ is a family of disjoint Borel sets and $\mu\left(\bigcup D_m \right) = \mu(\Omega_0) = 1$.
  In addition, for all $\omega \in D_m \subseteq B_{\sigma^m,\xi^m,\bar \rc^m}$ and $\rb \in \rM$ we have
  \begin{gather*}
    |\hat \xi^m_{\sigma^m}(\rb(\omega),\bar \rc^m(\omega))
    -
    \varphi(\ra(\omega),\rb(\omega))|
    < \varepsilon.
  \end{gather*}
  For each $m$ let $X_m  = \bP[D_m|\sF^\rcM] \in \sA^\rcM$ and let
  \begin{gather*}
    \psi(\ry) = \sum_m E \bigl[
    X_m \[ \hat \xi^m_{\sigma^m}(\ry,\bar \rc^m) \]
    \bigr].
  \end{gather*}
  Since $\sum X_m = 1$ this infinite sum converges uniformly to an $\rM$-definable predicate.
  We now claim that $\psi$ is $\varepsilon$-close to a $E\[\varphi\]$-definition for $\rp$.
  Indeed, for $\rb \in \rM$ we have $\[ \hat \xi^m_{\sigma^m}(\ry,\bar \rc^m) \] \in \sA^\rcM$, whereby
  \begin{gather*}
    \psi(\rb) =
    \sum_m \int_{D_m} \[ \hat \xi^m_{\sigma^m}(\rb,\bar \rc^m) \]
    \,d\mu.
  \end{gather*}
  We obtain
  \begin{align*}
    \bigl| \psi(\rb) - E\[\varphi(\ra,\rb)\] \bigr|
    &
    \leq \sum_m \int_{D_m}
    | \hat \xi^m_{\sigma^m}(\rb,\bar \rc^m)
    -
    \varphi(\ra,\rb) | \, d\mu
    < \varepsilon.
  \end{align*}
  We have shown that the predicate $\rb \mapsto E\[\varphi(\rx,\rb)\]^\rp$ is arbitrarily well approximated on $\rM_0$, and therefore on $\rM_0$, by an $\rM_0$-definable predicate.
  It follows that $\rp$ admits an $E\[\varphi\]$-definition.
  Since this holds for every type $\rp$ over a model the formula $E\[\varphi(\rx,\ry)\]$ is stable.

  The second assertion follows from the first using quantifier elimination down to formulae of the form $E\[\varphi\]$ (\fref{thm:TRTypes}), since continuous combinations of stable formulae are stable.
\end{proof}

\subsection{Dependence}

Recall that a formula $\varphi(\bar x,\bar y)$ is \emph{$\varepsilon$-independent} in a theory $T$ for some $\varepsilon > 0$ one can find in some model of $T$ an indiscernible sequence $(\bar b_n)_{n \in \bN}$ and $\bar a$ such that:
\begin{gather*}
  \bigvee_n \varphi(\bar a,\bar b_{2n}) + \varepsilon < \bigwedge_n \varphi(\bar a,\bar b_{2n+1}).
\end{gather*}
The formula $\varphi$ is \emph{independent} if it independent for some $\varepsilon > 0$.
The theory $T$ is \emph{dependent} if every formula is dependent, i.e., if every formula is $\varepsilon$-dependent for every $\varepsilon > 0$.

\begin{thm}
  A theory $T$ is dependent if and only if its randomisation $T^R$ is.
\end{thm}
\begin{proof}
  It is immediate to check that if $\varphi(\bar x,\bar y)$ is $\varepsilon$-independent in $T$ then $E\[\varphi(\bar \rx,\bar \ry)\]$ is $\varepsilon$-independent in $T^R$.
  The converse is \cite[Theorem~5.3]{BenYaacov:RandomVC}.
\end{proof}

For the converse, let us extend the so-called $TP_2$ to continuous logic:
\begin{dfn}
  We say that a theory $T$ has the \emph{tree property of the second kind}, or $TP_2$, if there exists a formula $\varphi(\bar x,\bar y)$, and in a model of $T$ an array $(\bar b_{n,m})_{n,m \in \bN}$, such that:
  \begin{enumerate}
  \item The sequences $I_n = (b_{n,m})_{m \in \bN}$ are mutually indiscernible, i.e., each is indiscernible over the others.
  \item The sequence of sequences $(I_n)_{n \in \bN}$ is indiscernible.
  \item The set of conditions $\{\varphi(\bar x,\bar b_{n,m}) = 0\}_{m \in \bN}$ is inconsistent for one, or equivalently all, $n$.
  \item The set of conditions $\{\varphi(\bar x,\bar b_{n,f(n)}) = 0\}_{n \in \bN}$ is consistent for one, or equivalently every, map $f\colon \bN \to \bN$.
  \end{enumerate}
\end{dfn}

Let us also recall that a theory is \emph{simple} (Shelah \cite{Shelah:SimpleUnstableTheories}) if every complete type over a set $A$ does not divide over a subset $A_0 \subseteq A$ with $|A_0| \leq |T|$.
Shelah \cite[Chapter~III.7]{Shelah:ClassificationTheory} proves that $TP_2$ implies both independence and non simplicity (and more), via the study of the relations between associated cardinal invariants.
Since this is proved for classical logic, let us give a quick argument why the same is true in continuous logic (since there is no treatment of simplicity in continuous logic as such, we refer the reader to \cite{BenYaacov:SimplicityInCats} for a treatment of simplicity in the even larger context of compact abstract theories).

\begin{prp}
  Assume $T$ has $TP_2$.
  Then $T$ is neither simple nor dependent.
\end{prp}
\begin{proof}
  By compactness we may extend the indiscernible sequence of lines to length $\kappa = |T|^+$, and find $\bar a$ such that $\varphi(\bar a,\bar b_{i,0}) = 0$ for all $i < \kappa$.
  Then $\tp(\bar a/\bar b_{<\kappa,0})$ divides over every subset of size $\leq |T|$, so $T$ is not simple.
  By compactness, there exists $\varepsilon > 0$ such that $\{\varphi(\bar x,\bar b_{n,m}) \leq \varepsilon\}_{m \in \bN}$ is inconsistent for all $n \in \bN$.
  Let $\bar c_n = \bar b_{n,0}$ for even $n$, and for odd $n$ let $\bar c_n = \bar b_{n,m}$ such that $\varphi(\bar a,\bar b_{n,m}) > \varepsilon$.
  Then $J = (\bar c_n)$ is indiscernible, and along with $\bar a$ witnesses that $T$ is independent.
\end{proof}

\begin{thm}
  \label{thm:IndependentRandomisation}
  Assume $T$ is independent.
  Then $T^R$ has the tree property of the second kind, and is in particular neither dependent nor simple.
\end{thm}
\begin{proof}
  After a few standard manipulations of the independent formula $\varphi$, we may assume that there is a model $\cM \vDash T$ and in there an indiscernible sequence $(b_m)$ as well as $a$ such that $\varphi(a,b_{2m}) = 0$ and $\varphi(a,b_{2m+1}) = 1$.
  Assuming $\cM$ to be saturated enough, it follows that for every $w \subseteq \bN$ there exists $a_w \in M$ such that $\varphi(a_w,b_m) = 0$ if $m \in w$ and $= 1$ otherwise.
  Let $\sA \vDash ARV$ and $\rcM = \cM^\sA \vDash T^R$, and let us identify members of $M$ with constant random variables in $\rM$.
  In $\sA$ we may find an array $(A_{n,m})$ of independent events of measure $\half$, which in fact forms an indiscernible set, and by \fref{cor:AuxiliarySortDefinability} it is an indiscernible set over $M$ (the constants).
  Let us consider the array $(b_nA_{n,m})_{n,m}$ in $\rM$, in which $b_n$ occurs repeatedly throughout the $n$th line $I_n = (b_nA_{n,m})_m$.
  Then the lines are mutually indiscernible, and form an indiscernible sequence $(I_n)_n$.

  Consider now the formula $\psi(\rx,\ry U) = d\bigl( \[\varphi(\rx,\ry)\], U \bigr)$.
  For $n \in \bN$ and $w \subseteq n$ let $B_{w,n} = \bigwedge_{i < n} A_{i,0}^{i \in w}$, where $A^\top = A$ and $A^\bot = \neg A$.
  By fullness we may construct $\ra_n$ such that for each $w \subseteq n$, $\ra_n = a_w$ on $B_{w,n}$, and observe that $\psi(\ra_n,b_i A_{i,0}) = 0$ for all $i < n$.
  In a saturated elementary extension we may therefore find $\ra$ such that $\psi(\ra,b_n A_{n,0}) = 0$ for all $n$.
  On the other hand, if $\psi(\ra',b_0 A_{0,0}) = 0$ then $\psi(\ra',b_0 A_{0,1}) = \half$.
  Thus $T^R$ has $TP_2$.
\end{proof}

\begin{qst}
  Say that a continuous theory $T$ has the \emph{strict order property (SOP)} if there exists a formula $\varphi(\bar x,\bar y)$ which defines a continuous pre-ordering with infinite $\varepsilon$-chains for some $\varepsilon > 0$, i.e., satisfies:
  \begin{itemize}
  \item \emph{Reflexivity:} $\varphi(\bar a,\bar a) = 0$.
  \item \emph{Transitivity:} $\varphi(\bar a,\bar c) \leq \varphi(\bar a,\bar b) + \varphi(\bar b,\bar c)$.
  \item \emph{Infinite $\varepsilon$-chain:}
    There exists $\varepsilon > 0$ and a sequence $(\bar a_n)_{n\in\bN}$ in a model of $T$ such that:
    \begin{gather*}
      \bigvee_{n<m} \varphi(\bar a_n,\bar a_m) + \varepsilon < \bigwedge_{n>m} \varphi(\bar a_n,\bar a_m).
    \end{gather*}
  \end{itemize}
  One can show that $T$ is unstable if and only if it is independent or has the strict order property.
  Indeed, a straightforward translation of the proof for classical first order theories, as can be found in Poizat \cite{Poizat:Cours}, would work, keeping in mind that every formula of the form $\varphi(x,x') = \sup_y \bigl( \psi(x,y) \dotminus \psi(x',y) \bigr)$ defines a continuous pre-ordering, in analogy with formulae of the form $\forall y \,\bigl( \psi(x,y) \to \psi(x',y) \bigr)$ in classical logic.

  \begin{enumerate}
  \item Assume $T$ is independent.
    Does $T^R$ has the strict order property?
  \item Alternatively, is it true that if $T$ does not have the strict order property then neither does $T^R$?
  \end{enumerate}
\end{qst}

\begin{cor}
  Randomisation cannot produce simple unstable theories: if $T^R$ is simple then it is in fact stable.
\end{cor}
\begin{proof}
  As in classical logic, the strict order property implies non simplicity, so a simple dependent theory is stable.
\end{proof}

\section{Lascar types}
\label{sec:LascarTypes}

\begin{dfn}
  Let $a$ and $b$ be two tuples, possibly infinite, in a structure
  $\cM$.
  We say that $d^L(a,b) \leq 1$ if in some (every) sufficiently
  saturated elementary extension $\cN \succeq \cM$ there exists an
  elementary sub-structure $\cN_0 \preceq \cN$ such that
  $a \equiv_{\cN_0} b$.
  We say that $d^L(a,b) \leq n$ if in some (every) sufficiently
  saturated elementary extension there
  exist $a_0 = a, a_1, \ldots, a_n = b$ such that
  $d^L(a_k,a_{k+1}) \leq 1$ for $k < n$.

  If $d^L(a,b) < \infty$, i.e.,
  if $d^L(a,b) \leq n$ for some $n$, then we say that
  $a$ and $b$ have the same \emph{Lascar type},
  in symbols $a \equiv^L b$.
\end{dfn}

The following is by now essentially folklore, and in any case quite easy.
\begin{fct}
  \label{fct:LascarType}
  \begin{enumerate}
  \item For every $n$, the relation $d^L(x,y) \leq n$ is
    a reflexive, symmetric type-definable relation.
  \item The relation $\equiv^L$ is the transitive closure of
    $d^L(x,y) \leq n$ for any $n > 0$.
    It is the finest bounded automorphism-invariant equivalence
    relation on the sort in question.
  \item If $d^L(a,b) \leq n$ in some sufficiently saturated structure
    $\cM$ then for every other tuple $a'$ there exists $b'$ such that
    $d^L(a'a,b'b) \leq n$.
  \end{enumerate}
\end{fct}

\begin{dfn}
  Let $(\rcM,\sA) \vDash T^R$.
  An \emph{$\sA$-type} in $(\rcM,\sA)$ is a complete type
  over a subset of $\sA$ which has a unique extension to a type over
  $\sA$.
  We define the \emph{$\sA$-type} of $\bar \ra \in \rM^n$ to be
  $\tp_\sA(\bar \ra) = \tp(\bar \ra/\sigma(\bar \ra))$.
\end{dfn}

\begin{lem}
  \label{lem:AType}
  A type $\rp(\bar \rx)$ over a subset of $\sA$ is an $\sA$-type if and only if it is equivalent to $\tp_\sA(\bar \ra)$ for some $\bar \ra$, if and only if it determines $\[ \varphi(\bar \ra) \]$ for every formula $\varphi(\bar x)$.
  It is then axiomatised by the set of all conditions of the form $\[ \varphi(\bar \rx) \]
  = \[ \varphi(\bar \ra) \]$.
  Moreover, $\tp_\sA(\bar \ra) = \tp_\sA(\bar \rb)$ if and only if, in the canonical representation, $\bar \ra(\omega) \equiv \bar \rb(\omega)$ for all $\omega$.
\end{lem}
\begin{proof}
  Easy, using \fref{cor:AuxiliarySortDefinability}.
\end{proof}

We may therefore write $\bar \ra \equiv_\sA \bar \rb$
to say that $\bar \ra$ and $\bar \rb$ have the same
$\sA$-type.
Similarly, if $\rp(\bar \rx)$ is an $\sA$-type we may write
$\[ \varphi(\bar \rx) \]^\rp \in \sA$ for the value
of $\[ \varphi(\bar \rx) \]$ as determined by $\rp$.

\begin{lem}
  \label{lem:ReductionToSameAType}
  Let $T$ be any theory, $(\rcM,\sA)$ a sufficiently saturated model
  of $T^R$.
  Let $\ra,\rb \in \rM$, $\ra \equiv \rb$.
  Then there exists $\rc \in \rM$ such that
  $d^L(\ra,\rc) \leq 1$ (so $\ra \equiv^L \rc$) and
  $\rc \equiv_\sA \rb$.
\end{lem}
\begin{proof}
  Let $\sA_0 \subseteq \sA$ be the sub-structure
  generated by $\[ \varphi \]$ where $\varphi$
  varies over all sentences.
  Then $\sA_0 \subseteq \dcl(\emptyset)$ in
  $(\rcM,\sA)$.
  Let $\Phi$ be the collection of all formulae $\varphi(x)$ with
  the appropriate variable.
  For $\varphi \in \Phi$ let
  $X_\varphi = \[ \varphi(\ra) \]$ and let
  $\bar X = (X_\varphi)_{\varphi\in\Phi}$.
  Define
  $\bar Y = (Y_\varphi)_{\varphi\in\Phi}
  = (\[ \varphi(\rb) \])_{\varphi\in\Phi}$
  similarly.
  Then by assumption $\bar X \equiv \bar Y$ in $(\rcM,\sA)$, whereby
  $\bar X \equiv_{\sA_0} \bar Y$.

  Let $(\rcM_1,\sA_1) \preceq (\rcM,\sA)$ be a small elementary
  sub-structure.
  Then necessarily $\sA_0 \subseteq \sA_1$.
  By our saturation assumption we may find
  $\sA_2 \subseteq \sA$ such that
  $\sA_2 \equiv_{\sA_0} \sA_1$ and
  $\sA_2 \ind_{\sA_0} \bar X,\bar Y$, both in the sense of
  $\sA$ (as a model of $ARV$).
  By \fref{cor:AuxiliarySortDefinability}
  we have $\sA_2 \equiv \sA_1$ in the structure
  $(\rcM,\sA)$ so again by saturation
  there is $\rcM_2 \subseteq \rcM$ such that
  $(\rcM_2,\sA_2) \equiv (\rcM_1,\sA_1)$, and in particular
  $(\rcM_2,\sA_2) \preceq (\rcM,\sA)$.

  By construction we have
  $\bar X \equiv_{\sA_2} \bar Y$ in the sense of $\sA$
  and applying \fref{cor:AuxiliarySortDefinability} again
  we obtain
  $\bar X \equiv_{(\rcM_2,\sA_2)} \bar Y$.
  Thus $d^L(\bar X,\bar Y) \leq 1$.
  By \fref{fct:LascarType}
  there is $\rc$ such that
  $d^L(\ra\bar X,\rc\bar Y) \leq 1$.
  In particular $\ra\bar X \equiv \rc\bar Y$
  whereby $Y_\varphi = \[ \varphi(\rc) \]$
  for all $\varphi(x)$.
  Thus $\rc \equiv_\sA \rb$ as desired.
\end{proof}

We now turn to consider the case where $\ra \equiv_\sA \rb$.
We shall require an additional technical result.

\begin{lem}
  \label{lem:InvariantIntegration}
  Let $\Omega$ be a set, $\tau\colon \Omega \to \Omega$ a bijection.
  Then there exists an integration functional
  $E$ on $\sA = [0,1]^\Omega$ which is moreover invariant under
  $\tau$:
  $E[X] = E[X \circ \tau]$ for all $X \in \sA$.
\end{lem}
\begin{proof}
  This is a special case of a general fact that if an amenable group
  $G$ (in our case, $(\bZ,+)$) acts on a space $\Omega$ then $\Omega$
  admits a $G$-invariant probability integration functional.
\end{proof}

\begin{lem}
  \label{lem:LascarAType}
  Let $T$ be any theory, $(\rcM,\sA) \vDash T^R$.
  Let $\ra,\rb \in \rM$, $\ra \equiv_\sA \rb$.
  Then $d^L(\ra,\rb) \leq 1$ (so in particular $\ra \equiv^L \rb$).
\end{lem}
\begin{proof}
  Let $(\rM,\sA)$ be the canonical representation of $(\rcM,\sA)$, based on $\sM = \{\cM_\omega\}_{\omega \in \Omega}$.
  Then for every $\omega \in \Omega$ we have $\ra(\omega) \equiv \rb(\omega)$ in $\cM_\omega$, so there exists an elementary extension $\cM'_\omega \succeq \cM_\omega$ and $h_\omega \in G_\omega = \Aut(\cM'_\omega)$ such that $h_\omega \ra(\omega) = \rb(\omega)$.
  Let $\bar G = \prod G_\omega$, $\bar h = (h_\omega)_\omega \in \bar G$.
  Let $\Omega' = \Omega \times \bar G$, and let $\pi\colon \Omega' \to \Omega$ be the projection on the first coördinate.
  By \fref{fct:PartialIntegrationFunctional} there exists an integration functional $E_1$ on $\sA_1 = [0,1]^\Omega$ which extends integration of Borel functions.
  The left action of $\bar h$ on $\bar G$ is bijective, so $[0,1]^{\bar G}$ admits an integration functional $E_G$ which is invariant under the left action of $\bar h$.
  Let $\sA' = [0,1]^{\Omega'}$, and for $X' \in \sA'$ define $E'[X'] = E\bigl[ \omega \mapsto E_G[ X'(\omega,\cdot) ]$, which we may also write as $E^\omega[ E_G^{\bar g}[ X'(\omega,\bar g) ]]$ or simply $E[ E_G[X] ]$.
  Then $E'$ is easily checked to be an integration functional.

  For $(\omega,\bar g) \in \Omega'$ let $\cM'_{(\omega,\bar g)} = \cM'_\omega$, thus obtaining a family $\sM'_{\Omega'} = \{\cM'_{\omega'}\}_{\omega'\in\Omega'}$ with $\cM_{\pi\omega'} \preceq \cM'_{\omega'}$.
  Let $\sigma_{\omega'}\colon \cM_{\pi\omega'} \hookrightarrow \cM'_{\omega'}$ denote this elementary inclusion.
  For $(\omega,\bar g) = (\omega,g_\zeta)_{\zeta \in \Omega} \in \Omega'$ define $\eta_{(\omega,\bar g)} = g_\omega \circ \sigma_{(\omega,\bar g)}\colon \cM_\omega \hookrightarrow \cM'_{(\omega,\bar g)}$, which is another elementary embedding.
  With $\rM' = \prod \sM'_{\Omega'}$, we obtain two maps $\sigma,\eta\colon \rM \to \rM'$, given by
  \begin{align*}
    (\sigma \rc)(\omega,\bar g) & = \sigma_{(\omega,\bar g)}(\rc(\omega)) = \rc(\omega), \\
    (\eta \rc)(\omega,\bar g) & = \eta_{(\omega,\bar g)}(\rc(\omega)) = g_\omega(\rc(\omega)).
  \end{align*}

  We are now in the situation described in \fref{cor:RandomFamilyExtension}.
  In particular, the triplet $(\rM',\sA',E')$ is a randomisation, and we obtain \emph{two} $\[\cdot\]$-embeddings, $[\sigma],[\eta]\colon
  (\rcM,\sA) \to (\widehat {\rcM'},\widehat \sA')$, where $[\sigma \rc]$ and $[\eta \rc]$ are the images in $\widehat{\rcM'}$ of $\sigma\rc$ and $\eta\rc$ defined above, and $[\sigma X] = [\eta X] = [X \circ \pi]$.
  By quantifier elimination for $T^R$ these embeddings are elementary.

  We claim that $[\sigma \ra] \equiv_{[\eta \rM]} [\sigma \rb]$.
  Indeed, let $\bar \rc \in \rM$, $\varphi(x,\bar y)$ any formula, and let $X = \< \varphi(\sigma\ra,\eta\bar \rc) \>$ and $Y = \< \varphi(\sigma\rb,\eta\bar \rc) \>$, both members of $[0,1]^{\Omega'}$.
  Fix $\omega \in \Omega$, and let $\bar g \in \bar G$ vary freely.
  Then:
  \begin{align*}
    X(\omega,\bar g)
    &
    = \varphi\bigl( \ra(\omega), \eta_{(\omega,\bar g)}\bar \rc(\omega) \bigr)
    \\ &
    = \varphi\bigl( h_\omega\ra(\omega), h_\omega\eta_{(\omega,\bar g)}\bar \rc(\omega) \bigr)
    \\ &
    = \varphi\bigl( \rb(\omega), \eta_{(\omega,\bar h\bar g)}\bar \rc(\omega) \bigr)
    \\ &
    = Y(\omega,\bar h\bar g).
  \end{align*}
  Since $E_G$ was chosen invariant under the left action of $\bar h$ on $\bar G$ we obtain that $E_G[ X(\omega,\cdot) ] = E_G[ Y(\omega,\cdot) ]$ for all $\omega$, whereby $E'[X] = E'[Y]$.
  We obtain
  \begin{gather*}
    E\[ \varphi([\sigma\ra],[\eta\bar \rc]) \]
    =
    E' \<\varphi(\sigma\ra,\eta\bar \rc) \>
    =
    E' \<\varphi(\sigma\rb,\eta\bar \rc) \>
    =
    E\[ \varphi([\sigma\rb],[\eta\bar \rc]) \],
  \end{gather*}
  proving our claim.
  Since $[\eta]$ is an elementary embedding we have $d^L([\sigma\ra],[\sigma\rb]) \leq 1$, and since $[\sigma]$ is an elementary embedding we conclude that $d^L(\ra,\rb) \leq 1$.
\end{proof}

\begin{thm}
  Let $T$ be any theory, $(\rcM,\sA) \vDash T^R$, $\ra,\rb \in \rM$, and let $\rA \subseteq \rM$ be any set of parameters.
  Then the following are equivalent:
  \begin{enumerate}
  \item $\ra \equiv^L_\rA \rb$.
  \item $\ra \equiv_\rA \rb$.
  \item $d^L_\rA(\ra,\rb) \leq 2$, where $d^L_\rA$ is defined as $d^L$, over parameters in $\rA$.
  \end{enumerate}
\end{thm}
\begin{proof}
  First of all we may name $\rA$ in the language (at no point did we assume that $T$ was complete), so we may assume that $\rA = \emptyset$.
  For the implication (ii) $\Longrightarrow$ (iii), just apply \fref{lem:ReductionToSameAType} followed by \fref{lem:LascarAType}.
  The implications (iii) $\Longrightarrow$ (i) $\Longrightarrow$ (ii) are standard and holds in arbitrary structures.
\end{proof}

\begin{cor}
  Let $\rcM \vDash T^R$, $\rA \subseteq \rM$.
  Let $\dcl^{eq,R}$ denote the definable closure in the sense of $(T^R)^{eq}$, and similarly for $\acl^{eq,R}$.
  Then $\dcl^{eq,R}(\rA) = \acl^{eq,R}(\rA)$ in $\rcM$.
\end{cor}
Notice that even though $\dcl^{eq,R}(\rA)$ and $\acl^{eq,R}(\rA)$ may contain imaginary elements in the sense $T^R$, the set $\rA$ is required to consist of real elements, i.e., elements coming from sorts of $T$.

\begin{cor}
  For every theory $T$, the theory $T^R$ is $G$-compact, which means that for every set of parameters $\rA$ and for every tuple length $\alpha$, the relation $\bar \ra \equiv^L_\rA \bar \rb$ between tuples of length $\alpha$ is type-definable over $\rA$.
\end{cor}
\begin{proof}
  Since the relation $d^L_\rA(\bar x,\bar y) \leq 2$ is type-definable.
\end{proof}

\bibliographystyle{begnac}
\bibliography{begnac}

\end{document}